\documentclass[xcolor=dvipsnames,svgnames,table,reqno]{amsart}

\input xy
\xyoption{all}
\usepackage{accents}
\usepackage{epsfig}
\usepackage{xcolor}
\usepackage{amsthm}
\usepackage{amssymb}
\usepackage{amsmath}
\usepackage{bbm}
\usepackage{amscd}
\usepackage{amsopn}
\usepackage{graphicx}
\usepackage{xspace}

\usepackage{hhline}
\usepackage{easybmat}
\usepackage{caption}   
\usepackage{relsize}

\usepackage{url}
\usepackage{enumitem, hyperref}\hypersetup{colorlinks}

\usepackage{stmaryrd}

\usepackage[T1]{fontenc}

\usepackage[color=yellow]{todonotes}

\colorlet{purpleB70}{blue!70!red}

\colorlet{orangeR65}{red!65!yellow}

\definecolor{red2}{HTML}{d41173}

\definecolor{neongreen}{HTML}{1bf702}

\definecolor{radicalred}{HTML}{FF355E}

\definecolor{denim}{HTML}{1560BD}

\definecolor{darkcyan}{rgb}{0.0, 0.55, 0.55}

\definecolor{cilek}{HTML}{FF43A4}

\definecolor{mor}{HTML}{9F00C5}

\definecolor{phlox}{rgb}{0.87, 0.0, 1.0}

\definecolor{fluorescentpink}{HTML}{FF1493}

\definecolor{napiergreen}{rgb}{0.16, 0.5, 0.0}

\definecolor{kellygreen}{rgb}{0.3, 0.73, 0.09}

\definecolor{parisgreen}{HTML}{ 50C878 }

\definecolor{palatinateblue}{rgb}{0.15, 0.23, 0.89}

\definecolor{ceruleanblue}{rgb}{0.16, 0.32, 0.75}

\definecolor{brandeisblue}{rgb}{0.0, 0.44, 1.0}

\definecolor{KLMblue}{HTML}{0FC0FC}

\definecolor{tealblue}{HTML}{2DFFC7}

\definecolor{cinnamon}{rgb}{0.82, 0.41, 0.12}

\definecolor{darkorange}{rgb}{1.0, 0.55, 0.0}

\definecolor{darktangerine}{rgb}{1.0, 0.66, 0.07}

\definecolor{deepcarrotorange}{rgb}{0.91, 0.41, 0.17}

\definecolor{internationalorange}{HTML}{FF4F00}

\definecolor{persimmon}{HTML}{EC5800}

\definecolor{pumpkin}{HTML}{FF7518}

\definecolor{darkred}{rgb}{1,0,0} 
\definecolor{darkgreen}{rgb}{0,0.8,0}
\definecolor{darkblue}{rgb}{0,0,1}

\hypersetup{colorlinks,
linkcolor=darkblue,
filecolor=darkgreen,
urlcolor=darkred,
citecolor=darkgreen}

\makeatletter
\def\reflb#1#2{\begingroup
    #2%
    \def\@currentlabel{#2}%
    \phantomsection\label{#1}\endgroup
}
\makeatother

\numberwithin{equation}{section}
\newtheorem {Theorem}{Theorem}
\numberwithin{Theorem}{section}

\newtheorem {Lemma}[Theorem]    {Lemma}
\newtheorem {Claim}[Theorem]    {Claim}

\newtheorem {Proposition}[Theorem]{Proposition}
\newtheorem {Corollary}[Theorem]{Corollary}
\theoremstyle{definition}
\newtheorem{Definition}[Theorem]{Definition}
\theoremstyle{remark}
\newtheorem{Remark}[Theorem]{Remark}
\newtheorem{Example}[Theorem]{Example}

\def    \eps    {\epsilon}
\def    \l    {\lambda}

\newcommand{\CA}{{\mathcal A}}

\newcommand{\CR}{{\mathcal R}}
\newcommand{\CS}{{\mathcal S}}
\newcommand{\CSI}{{\mathcal S}_{\mathit ind}}

\newcommand{\cCS}{{\check{\mathcal S}}}
\newcommand{\cCSI}{{\check{\mathcal S}}_{\mathit ind}}

\newcommand{\Ham}{{\mathit{Ham}}}
\newcommand{\tHam}{{\widetilde{\mathit{Ham}}}}

\newcommand{\id}{{\mathit id}}

\newcommand{\const}{{\mathit const}}

\newcommand{\hsi}{\hat{\sigma}}

\newcommand{\tH}{\tilde{H}}

\newcommand{\PP}{{\mathcal P}}
\newcommand{\bPP}{\bar{\mathcal P}}

\def    \nat    {{\natural}}
\def    \F      {{\mathbb F}}

\def    \C      {{\mathbb C}}
\def    \R      {{\mathbb R}}

\def    \Z      {{\mathbb Z}}
\def    \N      {{\mathbb N}}
\def    \Q      {{\mathbb Q}}

\def    \T      {{\mathbb T}}
\def    \CP     {{\mathbb C}{\mathbb P}}
\def    \RP     {{\mathbb R}{\mathbb P}}

\def    \12    {{\frac{1}{2}}}

\def    \p      {\partial}

\def    \im     {\operatorname{im}}
\def    \Imm    {\operatorname{Im}}

\def    \Sp     {\operatorname{Sp}}
\def    \tSp     {\operatorname{\widetilde{Sp}}}
\def    \U     {\operatorname{U}}
\def    \SU     {\operatorname{SU}}
\def    \PU     {\operatorname{PU}}

\def    \HF     {\operatorname{HF}}

\def    \H     {\operatorname{H}}

\def    \CF     {\operatorname{CF}}

\def    \QS    {\operatorname{QS}}

\def    \bPP     {\bar{\mathcal{P}}}
\def    \bx     {\bar{x}}

\def    \rx     {\mathring{x}}
\def    \by     {\bar{y}}

\def    \bl    {\bar{\lambda}}

\def    \vDelta    {\vec{\Delta}}

\def    \sgn  {\mathit{sgn}}
\def    \mult  {\mathit{mult}}
\def    \loop  {\mathit{loop}}
\def    \MUCZ  {\operatorname{\mu_{\scriptscriptstyle{CZ}}}}

\def  \hmu      {\operatorname{\hat{\mu}}}

\def    \s  {\operatorname{c}}

\def    \ssminus        {\smallsetminus}

\def    \hh    {\mathrm{h}}

\begin{document}


\setlength{\smallskipamount}{6pt}
\setlength{\medskipamount}{10pt}
\setlength{\bigskipamount}{16pt}






\title[Pseudo-Rotations vs.\ Rotations]{Pseudo-Rotations vs.\ Rotations}

\author[Viktor Ginzburg]{Viktor L. Ginzburg}
\author[Ba\c sak G\"urel]{Ba\c sak Z. G\"urel}

\address{BG: Department of Mathematics, University of Central Florida,
  Orlando, FL 32816, USA} \email{basak.gurel@ucf.edu}

\address{VG: Department of Mathematics, UC Santa Cruz, Santa Cruz, CA
  95064, USA} \email{ginzburg@ucsc.edu}



\date{\today} 

\thanks{The work is partially supported by NSF CAREER award
  DMS-1454342 and NSF grant DMS-1440140 through MSRI (BG) and Simons
  Collaboration Grant 581382 (VG)}


\begin{abstract} Continuing the study of Hamiltonian pseudo-rotations
  of projective spaces, we focus on the conjecture that the
  fixed-point data set (the actions and the linearized flows at
  one-periodic orbits) of a pseudo-rotation exactly matches that data
  for a suitable unique true rotation even though the two maps can
  have very different dynamics. We prove this conjecture in several
  instances, e.g., for strongly non-degenerate pseudo-rotations of
  $\CP^2$ with some notable exceptions, which we call ghost
  pseudo-rotations. The existence of ghost pseudo-rotations is a
  completely open question. The conjecture is closely related to the
  properties of the action and index spectra of pseudo-rotations, and
  ghost pseudo-rotations, if they exist, satisfy all known
  restrictions on the fixed-point data for pseudo-rotations but these
  data are distinctly different from the data for any rotation. The
  main new ingredient of the proofs is purely combinatorial and of
  independent interest. This is the index divisibility theorem
  connecting the divisibility properties of the Conley--Zehnder index
  sequence for the iterates of a map with the properties of its
  spectrum.
\end{abstract}

\maketitle


\tableofcontents


\section{Introduction and main results}
\label{sec:intro+results}

\subsection{Introduction}
\label{sec:intro}
In this paper we continue the study of Hamiltonian pseudo-rotations of
$\CP^n$ started in \cite{GG:PR} and focus on the conjecture that the
fixed point data (the actions and the linearized flow as an element of
$\tSp(2n)$ up to conjugation) of a pseudo-rotation exactly matches
that data for a suitable unique true rotation, i.e., an element of a
Hamiltonian circle action. We show that this is indeed the case in
many instances, e.g., for strongly non-degenerate pseudo-rotations of
$\CP^2$, but possibly with some exceptions. (Even for $\CP^1$ this is
ultimately a highly non-trivial fact, although in this case it easily
follows from known results; see Section \ref{sec:2D}.) As a
consequence, while the two maps can have very different dynamics, it
is unlikely that they can be distinguished by Floer theoretical and
even more generally by symplectic topological methods.

The existence of a matching rotation is a consequence of symplectic
topological restrictions on the fixed-point data of pseudo-rotations
and rather involved combinatorics, which works for, roughly speaking,
almost all but not all sets of the fixed-point data. In fact, already
for $\CP^2$, there do exist sets of the fixed-point data that meet
these restrictions but cannot come from a true rotation. In other
words, there are no known symplectic topological obstructions to the
existence of pseudo-rotations with the fixed-point data not matching
any rotation. We call them \emph{ghost pseudo-rotations}.

This fact points at the possibility that the connection between
rotations and pseudo-rotations is much weaker than currently
thought. For instance, ghost pseudo-rotations cannot come from any
known variant of the conjugation method. A similar phenomenon must
occur for all $\CP^{n>1}$, but when $n>2$ the combinatorics gets very
complicated and the fixed-point data of ghost pseudo-rotations become
difficult to be described explicitly. If ghost pseudo-rotations exist
for $\CP^n$, it is entirely possible that there are manifolds that
admit pseudo-rotations but not true rotations.

In the context of this paper, a \emph{pseudo-rotation} of $\CP^n$ is a
Hamiltonian diffeomorphism of $\CP^n$ with exactly $n+1$ periodic
points; this is the minimal possible number of periodic points by the
Arnold conjecture, \cite{Fl, Fo, FW}. The periodic points are then
necessarily the fixed points.  The term comes from low-dimensional
dynamics where it is used for maps of $S^2$ (or the closed disk $D^2$)
with exactly two (or one) periodic points, which are then also the
fixed points. More generally, one should think of a pseudo-rotation as
a Hamiltonian diffeomorphism with finite and minimal possible number
of periodic points. We refer the reader to \cite{GG:PR, GG:AI, Sh} for
a discussion of various ways to make this definition precise for other
symplectic manifolds and some other related questions. Here we only
mention that many, probably most, symplectic manifolds do not admit
Hamiltonian diffeomorphisms with finitely many periodic orbits (and in
particular pseudo-rotations) by the Conley conjecture; see
\cite{GG:survey, GG:Rev} and references therein.

Among pseudo-rotations of $\CP^n$ are \emph{true rotations} (or just
rotations for brevity) with finitely many periodic points (a generic
condition in this class), where by a true rotation we understand a
Hamiltonian diffeomorphism arising from the action of an element of
$\SU(n)$ on $\CP^n$. Once $\CP^n$ is identified with the quotient of
the unit sphere $S^{2n+1}\subset\C^{n+1}$, a true rotation is the
time-one map generated by a quadratic Hamiltonian $Q=\sum a_i|z_i|^2$,
for a suitable choice of linear coordinates $z_i$. A rotation is a
pseudo-rotation, i.e., it has exactly $n+1$ periodic points, if and
only if $a_i-a_j\not\in\Q$ for $i \neq j$. Otherwise, it has
infinitely many periodic orbits.

However, not every pseudo-rotation is conjugate to a true rotation.
Indeed, true rotations have simple, essentially trivial dynamics.
This is in general not the case for pseudo-rotations, and
pseudo-rotations occupy a special place among low-dimensional
dynamical systems. In \cite{AK}, Anosov and Katok constructed area
preserving diffeomorphisms $\varphi$ of $S^2$ with exactly three
ergodic invariant measures, the area form and the two fixed points, by
introducing what is now known as the conjugation method; see also
\cite{FK} and references therein. Such a diffeomorphism $\varphi$ is
automatically a pseudo-rotation. Indeed, $\varphi$ is area preserving
and hence Hamiltonian, and $\varphi$ has exactly two periodic orbits,
which are its fixed points.  Furthermore, $\varphi$ is ergodic,
necessarily has dense orbits, and thus is not conjugate to a true
rotation. As a consequence, the products $(S^2)^n$ also admit
pseudo-rotations which are not conjugate to true rotations. The
conjugation method can also be applied to construct dynamically
interesting pseudo-rotations of $\CP^n$ and other symplectic toric
manifolds in higher dimensions; see \cite{LRS}.

The study of the dynamics of pseudo-rotations in dimension two by
symplectic topological methods (finite energy foliations) was
initiated in \cite{Br:Annals, Br, BH}. In \cite{GG:PR}, Floer theory
and the results from \cite{GG:gaps, GG:hyperbolic, GK} were utilized
to investigate the dynamics of pseudo-rotations in higher
dimensions. Furthermore, connections between pseudo-rotations and the
symplectic topology of the underlying manifold (Gromov--Witten
invariants and the quantum Steenrod square) have been recently
explored in \cite{CiGG:IMRN, CiGG, Sh:new, Sh}.

\subsection{Results}
\label{sec:results} The main goal of the paper is to compare the
numerical invariants of periodic orbits (the action, the mean index
and Floquet multipliers) for rotations and pseudo-rotations of
$\CP^n$. Hypothetically, for every pseudo-rotation there exists
exactly one ``matching'' true rotation with precisely the same
numerical invariants. (In particular, by this conjecture, every
periodic orbit of a pseudo-rotation is elliptic and non-degenerate.)
For $n=1$, this is an easy consequence of rather standard, although
highly non-trivial, results; see Section \ref{sec:2D}. We prove the
conjecture for strongly non-degenerate pseudo-rotations of $\CP^2$,
but with a certain class of exceptions, and, under some additional
assumptions, in all dimensions. Note also that this conjecture is
consistent with the general expectation that every pseudo-rotation is
obtained from a true rotation by a variant of the conjugation method;
cf.\ \cite{Br:Annals, FK}.

The aforementioned exceptions -- ghost pseudo-rotations (see
Definition \ref{def:ghost})-- are extremely unlikely to
exist. However, if they do, this will be a strong indication that the
connection between pseudo-rotations and rotations is much more loose
than currently thought.

\subsubsection{Rotations and pseudo-rotations}
To state the results more precisely we need to recall several facts
about pseudo-rotations. These facts are essentially of symplectic
topological nature and go back to \cite{GG:gaps, GK}. (A more detailed
review is given in Section \ref{sec:background} and the relevant
definitions are recalled in Section \ref{sec:prelim}.) Here and
throughout the paper $\CP^n$ is equipped with the standard, up to a
factor, Fubini--Study symplectic structure $\omega$.

\begin{Definition}
\label{def:PR}
A pseudo-rotation of $\CP^n$ is a Hamiltonian diffeomorphism of
$\CP^n$ with exactly $n+1$ periodic points.
\end{Definition}

Most of our results concern strongly non-degenerate pseudo-rotations
$\varphi$. (Recall that $\varphi$ is called strongly non-degenerate if
all iterates $\varphi^k$, $k\in\N$, are non-degenerate.) However, it
is useful and illuminating to discuss the key facts in a more general
setting.

Let $\varphi=\varphi_H$ be a pseudo-rotation of $\CP^n$ generated by a
time-dependent Hamiltonian $H$, viewed as an element of the universal
covering $\tHam(\CP^n)$ of the group $\Ham(\CP^n)$ of Hamiltonian
diffeomorphisms. We associate two spectra to $\varphi$. One is the
standard action spectrum $\CS(H)$ comprising the actions of capped
one-periodic orbits of the Hamiltonian flow $\varphi_H^t$. For
instance, when $\varphi_H$ is a true rotation and $H=\sum a_i|z_i|^2$
the action spectrum $\CS(H)$ is the union of the sets $a_i+\lambda\Z$
where $\lambda$ is the integral of $\omega$ over $\CP^1$.  (Here and
throughout we identify $\CP^n$ with the quotient of the unit sphere
$\sum|z_i|^2=1$.) The second spectrum is the mean index spectrum
$\CSI(\varphi)$ formed by the mean indices of capped one-periodic
orbits. Furthermore, every point in each of the spectra is marked or
labeled by an integer. For a non-degenerate pseudo-rotation $\varphi$,
a point in either of the spectra is marked by $l\in\Z$ when the
Conley--Zehnder index of the corresponding orbit is $2l-n$. Thus we
have the marked spectra $\cCS(H)$ and $\cCSI(\varphi)$; these are
functions $\Z\to \R$ sending $l$ to the action or the mean index of
the orbit of index $2l-n$, since all integers of the same parity as
$n$ occur as indices. This construction extends to degenerate
pseudo-rotations; see Section \ref{sec:spectra}.

The first key fact we need is that for a pseudo-rotation
$\varphi=\varphi_H$ the two spectra agree up to a factor and a shift:
$$
\cCS(H)=\frac{\lambda}{2(n+1)}\cCSI(\varphi)+\const,
$$
where \emph{const} depends on $H$. Thus by adding a constant to $H$ we
can ensure that
$$
\cCS(H)=\frac{\lambda}{2(n+1)}\cCSI(\varphi).
$$
When $\varphi_H$ is a rotation and $H$ is a quadratic form, this is
the condition that $\sum a_i=0$. Let us denote the points labeled by
$0,\ldots, n$ in $\cCSI(\varphi)$ by $\Delta_0,\ldots,\Delta_n$. One
can show that $\sum \Delta_i=0$ when $\varphi$ is a true rotation; see
Lemma \ref{lemma:new_sum}. In general, let us call a pseudo-rotation
meeting this requirement \emph{balanced}; see Definition
\ref{def:balanced}. This is a condition on the fixed-point data or, to
be more precise, on the mean indices of the one-periodic orbits of
$\varphi$. It was conjectured in \cite{GK} that every pseudo-rotation
is balanced. In Section \ref{sec:PR} we prove this conjecture for
strongly non-degenerate pseudo-rotations of $\CP^2$ which are not
ghost pseudo-rotations:

\begin{Theorem}
  \label{thm:balanced-4D-intro}
  Let $\varphi$ be a strongly non-degenerate pseudo-rotation of
  $\CP^2$. Then all fixed points of $\varphi$ are elliptic and, unless
  it is a ghost pseudo-rotation in the sense of Definition
  \ref{def:ghost}, $\varphi$ is balanced.
\end{Theorem}

To define ghost pseudo-rotations $\varphi$ of $\CP^2$, consider the
first Krein type eigenvalues of the linearization $D\varphi$ at its
three fixed points; see Section \ref{sec:index} or \cite{Ab,SZ,Lo} for
the definition.  We will assume that $\varphi$ is strongly
non-degenerate and therefore these eigenvalues are unit complex
numbers (not equal to 1) by Theorem \ref{thm:balanced-4D-intro}. Thus
to each of the three fixed points we have associated a pair of unit
complex numbers which we will call \emph{Floquet multipliers}.  By
definition, $\varphi$ is a ghost pseudo-rotation when its Floquet
multipliers satisfy certain very restrictive constraints:

\begin{Definition}
  \label{def:ghost}
  A strongly non-degenerate pseudo-rotation $\varphi$ of $\CP^2$ is
  said to be a \emph{ghost pseudo-rotation} if its fixed points have
  Floquet multipliers of the form $(\lambda,\lambda)$,
  $(\lambda,\eta)$ and $(\eta,\eta)$, where $\lambda$ and $\eta$ are
  unit complex numbers and $\lambda=\bar{\eta}^2$.
\end{Definition}

It is easy to see that no true rotation can have the same Floquet
multipliers as a ghost pseudo-rotation. Deferring a more detailed
discussion of ghost pseudo-rotations to Section \ref{sec:Ghosts}, here
we only mention that one can expect all Floquet multipliers to be
distinct for a generic pseudo-rotation. Thus ghost pseudo-rotations,
if they exist, are probably extremely non-generic even in the class of
pseudo-rotations. However, one has to be careful with this kind of
assertions, since, for instance, the Floquet multipliers of a
pseudo-rotation must satisfy a resonance relation (see Section
\ref{sec:spectra} or \cite{GK}), although this relation need not have
the specific form from Definition \ref{def:ghost}. It is also worth
keeping in mind that the Floquet multipliers do not fully determine
the entire set of the fixed-point data: e.g., the mean indices and
indices are determined only modulo 2.

The proof of Theorem \ref{thm:balanced-4D-intro} relies on symplectic
topological restrictions on the fixed-point data of pseudo-rotations
$\varphi$, established in \cite {GG:gaps,GK}. Then a purely
combinatorial argument shows that $\varphi$ must be balanced unless it
is a ghost pseudo-rotation. In the latter case, there exist
fixed-point data compatible with the Floquet multipliers and meeting
those restrictions such that $\varphi$ is not balanced.

Returning to the comparison of rotations and pseudo-rotations note
that for a rotation $\varphi$ the spectrum $\CSI(\varphi)$ completely
determines the map $\varphi$ as an element of $\tHam(\CP^n)$; see
Section \ref{sec:R}. Then, as is easy to see, for a balanced strongly
non-degenerate pseudo-rotation $\varphi$, there exists a unique
rotation $R_\varphi$, called the \emph{matching} rotation, such that
$\CSI(\varphi)=\CSI(R_\varphi)$ and, moreover,
$\cCSI(\varphi)=\cCSI(R_\varphi)$; see Section \ref{sec:PR}. Thus we
have a one-to-one correspondence between the capped one-periodic
orbits of $\varphi$ and $R_\varphi$, where the corresponding orbits
$\bx_i$ of $\varphi$ and $\by_i$ of $R_\varphi$ have equal
Conley--Zender indices, mean indices and, up to a shift, actions. It
turns out that under some additional conditions the corresponding
orbits $x_i$ of $\varphi$ and $y_i$ of $R_\varphi$ have equal unit
spectra $\sigma(x_i)$ and $\sigma(y_i)$ of the linearized return
maps. For instance, we have the following result proved in Section
\ref{sec:PR}:

\begin{Theorem}
\label{thm:matching2-intro}
Let $\varphi$ be a balanced strongly non-degenerate pseudo-rotation of
$\CP^n$ and let $R_\varphi$ be its matching rotation. Assume that for
every one-periodic orbit $y_i$ of $R_\varphi$ all unit eigenvalues at
$y_i$ (i.e., the elements of $\sigma(y_i)$) are distinct and
$\sigma(y_i)\cap\sigma(y_j)=\emptyset$ for any pair $i\neq j$. Then
$\sigma(x_i)=\sigma(y_i)$.
\end{Theorem}

As a consequence, under the conditions of the theorem, the fixed
points of $\varphi$ are elliptic and all iterates $\varphi^k$ are
balanced. Moreover, for $\CP^2$, the assertion of the theorem holds
for any strongly non-degenerate pseudo-rotation without additional
assumptions on the spectra $\sigma(y_i)$; see Theorem
\ref{thm:matching1} and Corollary \ref{cor:matching-4D}. Proofs of
these results also rely on symplectic topological restrictions on the
fixed-point data of pseudo-rotations $\varphi$ and combinatorial
arguments. Higher dimensional analogues of ghost pseudo-rotations do
not arise in this case due to the assumption on the eigenvalues and/or
the condition that $\varphi$ is balanced.

\subsubsection{Index Analysis}
The key new ingredient in the proofs of these results is essentially
combinatorial. This is Theorem \ref{thm:index-div} (Index Divisibility
Theorem) relating the behavior of the Conley-Zehnder index under
iterations and the spectrum of a linear map. Although this result
plays a purely technical role in the paper, the theorem and its proof
are of independent interest. Referring the reader to Sections
\ref{sec:index} and \ref{sec:index-div-pf} for details, here we only
briefly outline the underlying idea.

Consider an element $\Phi\in\tSp(2n)$ which we require to be strongly
non-degenerate (i.e., all iterates $\Phi^k$, $k\in\N$, are
non-degenerate). Then we have the sequence of the Conley--Zehnder
indices $\mu_k:=\mu\big(\Phi^k\big)$ defined; see, e.g., \cite{Lo,
  SZ}. We denote by $\mu'_k$ the derivative or the index jump
sequence: $\mu'_k:=\mu_{k+1}-\mu_k$. Furthermore, let us decompose
$\Phi$ as the product of a loop $\phi$ and the direct sum of three
short paths: an elliptic path $\Phi_e$, a positive hyperbolic path
$\Phi_h$, and a negative hyperbolic path $\Phi_{-h}$. The requirement
that the paths are short, making this decomposition unique up to
homotopy, is explicitly spelled out in Section \ref{sec:index}. Denote
by $\sigma_+(\Phi)$ the part of the spectrum of $\Phi_e$ lying in the
upper half plane and, for $\lambda\in\sigma_+(\Phi)$, let
$\sgn_\lambda(\Phi)$ be its signature. Furthermore, set $\loop(\Phi)$
to be the mean index of $\phi$ and $\mult_{-1}(\Phi)$ to be half of
the dimension of the domain of $\Phi_{-h}$.

\begin{Theorem}[Index Divisibility]
  \label{thm:index-div-intro} Fix $l\in\N$.  The derivative sequence
  $\mu'_k$ is divisible by $2l$ if and only if the following two
  conditions are satisfied:
\begin{itemize}
\item[\rm(i)] $2l \mid \big( \loop(\Phi)+\mult_{-1}(\Phi) \big)$,
\item[\rm(ii)] $l \mid \sgn_{\lambda}(\Phi)$ for all
  $\lambda\in\sigma_+(\Phi)$.
\end{itemize}
\end{Theorem}

One consequence of the index divisibility theorem (Corollary
\ref{cor:index-spec}) is that each of the sequences $\mu_k$ and
$\mu'_k$ completely determines the spectrum $\sigma_+(\Phi)$ together
with signatures except for the eigenvalues with zero signature and
that the jump sequence determines the index sequence. The proof of the
theorem, given in Section \ref{sec:index-div-pf}, is geometrical and
relies on the properties of a certain cycle associated with $\Phi$ in
the torus $\T^r$, where $r$ is the number of distinct elements in
$\sigma_+(\Phi)$.

\medskip

\noindent{\bf Acknowledgements.} The authors are grateful to the
referee for useful remarks and suggestions and for spotting an error
in the first version of this paper.  Parts of this work were carried
out during the first author's visit to NCTS (Taipei, Taiwan) and while
the second author was in residence at MSRI, Berkeley, CA, during the
Fall 2018 semester.  The authors would like to thank these institutes
for their warm hospitality and support.

\section{Preliminaries}
\label{sec:prelim}
The goal of this section is to set notation and conventions, give a
brief review of Hamiltonian Floer homology and several other notions
from symplectic geometry used in the paper. The reader may consider
consulting this section only as necessary.

\subsection{Conventions and notation}
\label{sec:conventions}
Throughout the paper the underlying symplectic manifold
$(M^{2n},\omega)$ will be $\CP^n$ equipped with the standard
Fubini--Study symplectic form $\omega$. It will be convenient however
to vary the normalization of this form and so we set
$\lambda=\left<[\omega],\CP^1\right>$. In other words, $\lambda$ is
the positive generator of $\left< [\omega],\pi_2(M)\right>\subset \R$,
the rationality constant. For the standard Fubini--Study normalization
$\lambda=\pi$. Recall also that the minimal Chern number $N$, i.e.,
the positive generator of $\left< c_1(TM),\pi_2(M)\right>$, is $n+1$
for $\CP^n$.

All Hamiltonians $H$ are assumed to be $k$-periodic in time, i.e.,
$H\colon S^1_k\times M\to\R$, where $S^1_k=\R/k\Z$ and $k\in\N$.  When
the period $k$ is not specified, it is equal to one and
$S^1=S^1_1=\R/\Z$. We set $H_t = H(t,\cdot)$ for $t\in S^1_k$. The
Hamiltonian vector field $X_H$ of $H$ is defined by
$i_{X_H}\omega=-dH$. The (time-dependent) flow of $X_H$ is denoted by
$\varphi_H^t$ and its time-one map by $\varphi_H$. Such time-one maps
are referred to as \emph{Hamiltonian diffeomorphisms}. A one-periodic
Hamiltonian $H$ can always be treated as $k$-periodic, which we will
then denote by $H^{\nat k}$ and, abusing terminology, call
$H^{\nat k}$ the $k$-th iteration of $H$. \emph{Throughout the paper,
  all Hamiltonian diffeomorphisms are viewed as elements of the
  universal covering $\tHam(\CP^n)$ of the group $\Ham(\CP^n)$ of
  Hamiltonian diffeomorphisms.}

Let $x\colon S^1_k\to M$ be a contractible loop. A \emph{capping} of
$x$ is an equivalence class of maps $A\colon D^2\to M$ such that
$A|_{S^1_k}=x$. (Here $D^2\subset \C$ is the unit disk and
$S^1_k=\R/k\Z$ is identified with $\p D^2=S^1$ via the map
$t\mapsto e^{2\pi\sqrt{-1}t/k}$.) Two cappings $A$ and $A'$ of $x$ are
equivalent if the integrals of $\omega$ (and hence of $c_1(TM)$) over
the sphere obtained by attaching $A$ to $A'$, where $A'$ is taken with
the opposite orientation, are equal to zero. A capped closed curve
$\bar{x}$ is, by definition, a closed curve $x$ equipped with an
equivalence class of cappings. In what follows, the presence of
capping is always indicated by a bar.

The action of a Hamiltonian $H$ on a capped closed curve
$\bar{x}=(x,A)$ is
$$
\CA_H(\bar{x})=-\int_A\omega+\int_{S^1} H_t(x(t))\,dt.
$$
The space of capped closed curves is a covering space of the space of
contractible loops, and the critical points of $\CA_H$ on this space
are exactly the capped one-periodic orbits of $X_H$. The \emph{action
  spectrum} $\CS(H)$ of $H$ is the set of critical values of
$\CA_H$. This is a zero measure set; see, e.g., \cite{HZ}. (Sometimes,
abusing notation, we will use the notation $\CS(\varphi)$, where
$\varphi=\varphi_H$, with the understanding that $H$ is fixed in the
background.)

These definitions extend to $k$-periodic orbits and Hamiltonians in an
obvious way. Clearly, the action functional is homogeneous with
respect to iteration:
$$
\CA_{H^{\nat k}}\big(\bx^k\big)=k\CA_H(\bx),
$$
where $\bx^k$ is the $k$-th iteration of the capped orbit $\bx$.  We
denote the set of $k$-periodic orbits of $H$ by $\PP_k(H)$. The set of
all periodic orbits will be denoted by $\PP(H)$. An un-iterated
periodic orbit is called \emph{simple}.

A periodic orbit $x$ of $H$ is said to be \emph{non-degenerate} if the
linearized return map $D\varphi_H \colon T_{x(0)}M\to T_{x(0)}M$ has
no eigenvalues equal to one. A Hamiltonian $H$ is non-degenerate if
all its one-periodic orbits are non-degenerate and $H$ is
\emph{strongly non-degenerate} if all periodic orbits of $H$ (of all
periods) are non-degenerate.

Let $\bar{x}=(x,A)$ be a non-degenerate capped periodic orbit.  The
\emph{Conley--Zehnder index} $\mu(\bar{x})\in\Z$ is defined, up to a
sign, as in \cite{Sa, SZ}. In this paper, we normalize $\mu$ so that
$\mu(\bar{x})=n$ when $x$ is a non-degenerate maximum (with trivial
capping) of an autonomous Hamiltonian with small Hessian. The
\emph{mean index} $\hmu(\bx)\in\R$ measures, roughly speaking, the
total angle swept by certain unit eigenvalues of the linearized flow
$D\varphi^t_H|_x$ with respect to the trivialization associated with
the capping; see \cite{Lo, SZ}. The mean index is defined even when
$x$ is degenerate and depends continuously on $H$ and $\bx$ in the
obvious sense.  Furthermore,
\begin{equation}
\label{eq:mean-cz}
\big|\hmu(\bx)-\mu(\bx)\big|\leq n
\end{equation} 
and the inequality is strict when $x$ is non-degenerate or even weakly
non-degenerate, i.e., at least one of the eigenvalues is different
from one, \cite{SZ}. To make sense of \eqref{eq:mean-cz} when $x$ is
degenerate, one has to use one of the extensions of $\mu$ to
degenerate orbits. This can be, for instance, the Robbin--Salamon
index, \cite{RS}, or the upper or lower semi-continuous extension;
cf.\ \cite{GG:convex}. However, the extension which is more useful for
us is what we call the LS-index defined in Section
\ref{sec:background}. It is specific to $\CP^n$ and non-local.  The
mean index is homogeneous with respect to iteration:
$\hmu\big(\bx^k\big)=k\hmu(\bx)$.

The index $\mu(x)$ of an un-capped orbit $x$ is well defined as an
element of $\Z/2N\Z$. Likewise, the mean index $\hmu(x)$ is well
defined as an element of $\R/2N\Z$.

Throughout the paper, by the \emph{fixed-point data} of $\varphi$ we
mean the map which assigns to a capped one-periodic orbit $\bx$ the
linearized flow along $\bx$, up to conjugation in $\tSp(2n)$, and the
action of $\bx$. We will usually only need the numerical parts of this
data set, e.g., the eigenvalues of the return map and the
Conley--Zehnder and mean indices of the orbits.

\subsection{Floer homology and spectral invariants}
\label{sec:FH}
In this subsection, we very briefly discuss, mainly to set notation,
spectral invariants and Floer homology. We refer the reader to, e.g.,
\cite{GG:gaps, MS, Sa, SZ} for detailed accounts and additional
references.

Fix a ground field $\F$. For a non-degenerate Hamiltonian $H$ on $M$
we denote the filtered Floer complex of $H$ by $\CF^{(a,\, b)}_m(H)$,
where $-\infty\leq a < b \leq \infty$ and $a$ and $b$ are not in
$\CS(H)$. As a vector space, $\CF^{(a,\, b)}_m(H)$ is formed by finite
linear combinations
$$ 
\sigma=\sum_{\bar{x}\in \bPP(H)} \sigma_{\bar{x}}\bar{x}, 
$$
where $\sigma_{\bar{x}}\in\F$ and $\mu(\bar{x})=m$ and
$a<\CA_H(\bx)<b$. (Since $M$ is monotone we do not need to take a
completion here.) The Floer differential ``counts'' the
$L^2$-anti-gradient trajectories of the action functional and the
resulting homology, the \emph{filtered Floer homology} of $H$, is
denoted by $\HF^{(a,\, b)}_*(H)$ and by $\HF_*(H)$ when
$(a,\,b)=(-\infty,\,\infty)$. The degree of a class
$\alpha\in \HF^{(a,\, b)}_*(H)$ is denoted by~$|\alpha|$.

The definition of the Floer homology extends by continuity to all, not
necessarily non-degenerate, Hamiltonians.  Namely, let $H$ be an
arbitrary (one-periodic in time) Hamiltonian on $M$ and let the
end-points $a$ and $b$ of the action interval be outside $\CS(H)$. By
definition, we set
$$
\HF^{(a,\, b)}_*(H)=\HF^{(a,\, b)}_*(\tH),
$$
where $\tH$ is a non-degenerate, small perturbation of $H$. It is easy
to see that the right-hand side is independent of $\tH$ once $\tH$ is
sufficiently $C^2$-close to $H$.

The notion of \emph{local Floer homology} goes back to the original
work of Floer and it has been revisited a number of times since
then. Here we only briefly recall the definition following mainly
\cite{Gi:CC, GG:gaps, GG:gap} where the reader can find a much more
thorough discussion and further references.

Let $x$ be an isolated one-periodic orbit of a Hamiltonian
$H\colon S^1\times M\to \R$. The local Floer homology $\HF(x)$ is the
homology of the Floer complex generated by the orbits $x$ splits into
under a $C^2$-small non-degenerate perturbation of $H$ near $x$. This
homology is well-defined, i.e., independent of the perturbation. The
homology $\HF(x)$ is only relatively graded and to fix an absolute
grading one can pick a trivialization of $TM$ along $x$. This can be
done by using, for instance, a capping of $x$ and in this case we
write $\HF_*(\bx)$.  For instance, if $x$ is non-degenerate and
$\MUCZ(\bx)=m$, we have $\HF_l(\bx)=\F$ when $l=m$ and $\HF_l(\bx)=0$
otherwise. Moreover, $\HF(\bx)$ is supported in the interval
$[\hmu(\bx)-n,\, \hmu(\bx)+n]\subset \R$, i.e., $\HF_l(\bx)=0$ when
$l\in\Z$ is outside this interval, \cite{GG:gap}. The construction of
$\HF(x)$ is local: it requires $H$ to be defined only in a
neighborhood of $x$.

The total Floer homology is independent of the Hamiltonian and, up to
a shift of the grading and the effect of recapping, is isomorphic to
the homology of $M$. More precisely, we have
$$
\HF_*(H)\cong \H_ {*+n}(M;\F)\otimes \Lambda
$$
as graded $\Lambda$-modules, where $\Lambda$ is a suitably defined
Novikov ring. For instance, let $H$ be a Hamiltonian on $\CP^n$. Then
$\HF_m(H)=\F$ when $m$ has the same parity as $n$ and $\HF_m(H)=0$
otherwise. To see this, one can just take a non-degenerate quadratic
Hamiltonian as $H$ and observe that all fixed points of $\varphi_H$
are elliptic and hence their indices have the same parity as $n$. (In
particular, the Floer differential vanishes.)

The machinery of Hamiltonian \emph{spectral invariants} was developed
in its present Floer--theoretic form in \cite{Oh:constr, Sc}, although
the first versions of the theory go back to \cite{HZ, Vi:gen}. Action
carriers were introduced in \cite{GG:gaps} and then studied in
\cite{CGG, GG:nm}.

Let $H$ be a Hamiltonian on a closed monotone (or even rational)
symplectic manifold $M^{2n}$. The \emph{spectral invariant} or
\emph{action selector} $\s_\alpha$ associated with a class
$\alpha\in \HF_*(H)$ is defined as
$$
\s_\alpha(H)= \inf\{ a\in \R\ssminus \CS(H)\mid \alpha \in \im(i^a)\}
=\inf\{ a\in \R\ssminus \CS(H)\mid j^a\left( \alpha \right)=0\},
$$
where $i^a\colon \HF_*^{(-\infty,\,a)}(H)\to \HF_*(H)$ and
$j^a\colon \HF_*(H)\to\HF_*^{(a,\, \infty)}(H)$ are the natural
``inclusion'' and ``quotient'' maps. It is easy to see that
$\s_\alpha(H)>-\infty$ when $\alpha\neq 0$ and one can show that
$\s_\alpha(H)\in S(H)$. In other words, there exists a capped
one-periodic orbit $\bx$ of $H$ such that
$\s_\alpha(H)=\CA_H(\bx)$. As an immediate consequence of the
definition,
$$
\s_\alpha(H+a(t))=\s_\alpha(H)+\int_{S^1}a(t)\,dt,
$$
where $a\colon S^1\to\R$.

Spectral invariants have several important properties.  For instance,
the function $\s_\alpha$ is homotopy invariant:
$\s_\alpha(H)=\s_\alpha(K)$ when $\varphi_H=\varphi_K$ in $\tHam(M)$
and $H$ and $K$ have the same mean value. Furthermore, it is
sub-additive, monotone and Lipschitz in the $C^0$-topology as a
function of $H$.

When $H$ is non-degenerate, the action selector can also be evaluated
as
$$
\s_\alpha(H)=\inf_{[\sigma]=\alpha}\s_\sigma(H),
$$
where for a Floer chain $\sigma$ (in this case a cycle), we set
\begin{equation}
\label{eq:cycle-action}
\s_\sigma(H)
=\max\big\{\CA_H(\bx)\,\big|\, \sigma_{\bx} \neq 0\big\}\text{ for }
\sigma=\sum\sigma_{\bx} \bx\in\CF_{|\alpha|}(H).
\end{equation}
The infimum here is obviously attained, since $M$ is rational and thus
$\CS(H)$ is closed.  Hence there exists a cycle
$\sigma=\sum\sigma_{\bx} \bx\in\CF_{|\alpha|}(H)$, representing the
class $\alpha$, such that $\s_\alpha(H)=\CA_H(\bx)$ for an orbit $\bx$
entering $\sigma$. In other words, $\bx$ maximizes the action on
$\sigma$ and the cycle $\sigma$ minimizes the action over all cycles
in the homology class $\alpha$. Such an orbit $\bx$ is called a
\emph{carrier} of the action selector. This is a stronger requirement
than just that $\s_\alpha(H)=\CA_H(\bx)$ and $\mu(\bx)=|\alpha|$. When
$H$ is possibly degenerate, a capped one-periodic orbit $\bx$ of $H$
is a carrier of the action selector if there exists a sequence of
$C^2$-small, non-degenerate perturbations $\tH_i\to H$ such that one
of the capped orbits $\bx$ splits into is a carrier for $\tH_i$. An
orbit (without capping) is said to be a carrier if it turns into one
for a suitable choice of capping.

It is easy to see that a carrier necessarily exists, but in general is
not unique. However, it becomes unique when all one-periodic orbits of
$H$ have distinct action values.

As consequence of the definition of the carrier and continuity of the
action and the mean index, we have
$$
\s_\alpha(H) =
\CA_H(\bx)\text{ and } \big|\hmu(\bx)-|\alpha|\big| \leq n,
$$
and the inequality is strict when $x$ is weakly non-degenerate.
Furthermore, a carrier $\bx$ for $\s_\alpha$ is in some sense
homologically essential. Namely, $\HF_{|\alpha|}(\bx)\neq 0$, provided
that all one-periodic orbits of $H$ are isolated; cf. \cite[Lemma
3.2]{GG:nm}.

\section{Background results on pseudo-rotations}
\label{sec:background}

\subsection{Action and index spectra}
\label{sec:spectra}
In this section, we briefly recall several symplectic topological
results on pseudo-rotations of projective spaces, essential for our
purposes. A much more detailed treatment can be found in \cite[Sect.\
3]{GG:PR}.

Let $\varphi=\varphi_H$ be a pseudo-rotation of $\CP^n$, which we do
not assume to be non-degenerate. Denote by $\alpha_l$ the generator in
$\HF_{2l-n}(H)=\F$, $l\in\Z$, and let $\bx_{l}\in\bPP_1(H)$ be an
action carrier for $\alpha_l$. We will write
$\s_l:=\s_{\alpha_l}$. Then, in particular,
$$
\s_{l}(H)=\CA_H(\bx_l)\textrm{ and } \HF_{2l-n}(\bx_l)\neq 0,
$$
and hence $\HF(x)\neq 0$ for all $x\in\PP(H)$. When $\varphi$ is
non-degenerate, we have
\begin{equation}
  \label{eq:index-bxi}
  \mu(\bx_l)=2l-n.
\end{equation}

\begin{Theorem}[Lusternik--Schnirelmann Inequalities, \cite{GG:gaps}]
\label{thm:bijection}
For every $l\in \Z$ the action carrier $\bx_l$ is unique and the
resulting map
\begin{equation}
\label{eq:map-x_i}
\Z\to\bPP_1(H), \quad l\mapsto\bx_l
\end{equation}
is a bijection. Furthermore, the map
\begin{equation}
\label{eq:map-s}
\Z\to\CS(H), \quad l\mapsto \s_{l}(H)=\CA_H(\bx_l)
\end{equation}
is strictly monotone, i.e., $l>l'$ if and only if
$\CA_H(\bx_l)>\CA_H(\bx_{l'})$.
\end{Theorem}

One important consequence of the theorem is that distinct capped
one-periodic orbits of $\varphi_H$ necessarily have different actions.

The proof of Theorem \ref{thm:bijection} relies on a version of
Lusternik--Schnirelmann theory for action selectors. The theorem
allows us to extend the notion of the Conley--Zehnder index to capped
one-periodic orbits $\bx$ of degenerate pseudo-rotations by setting
$\mu(\bx)=2l-n$ for $\bx=\bx_l$. We will call $\mu(\bx)$ the
\emph{LS-index}. When $\bx$ is non-degenerate this is just the
ordinary Conley--Zehnder index. Without non-degeneracy, the LS-index
is a global rather than local invariant. However, it has many of the
expected properties of the Conley--Zehnder index, e.g.,
$\HF_{\mu(\bx)}(\bx)\neq 0$, the LS-index satisfies \eqref{eq:mean-cz}
and, when $n=1$, $\HF(\bx)$ is concentrated in only one degree which
is $\mu(\bx)$, \cite{GG:gap}. With this notion in mind, Theorem
\ref{thm:bijection} can be rephrased as that the ordering of
$\bPP_1(H)$ by the LS-index agrees with that by the action.

\begin{Theorem}[Action--Index Resonance Relations, \cite{GG:gaps}]
\label{thm:act-index}
For every $\bx\in\bPP_1(H)$, we have
\begin{equation}
\label{eq:act-index2}
\CA_H(\bx)=\frac{\lambda}{2(n+1)}\hmu(\bx)+\const,
\end{equation}
where $\const$ is independent of $x$.
\end{Theorem}

This theorem has been extended to some other symplectic manifolds and
a broader class of Hamiltonian diffeomorphisms; see \cite{CGG}. It
also has an analog for Reeb flows, \cite[Sect.\ 6.1.2]{GG:convex}.

The \emph{marked action spectrum} $\cCS(H)$ is, by definition, the
monotone bijection
$$
\cCS\colon
\Z\stackrel{\eqref{eq:map-x_i}}{\longleftrightarrow}\bPP_1(H)
\stackrel{\eqref{eq:map-s}}{\longrightarrow}\CS(H),
$$
i.e., $\cCS(H)$ is simply the spectrum $\CS(H)$ with its points
labeled by $\Z$ (essentially the indices) or, equivalently, by
$\bPP_1(H)$. In a similar vein, the \emph{marked index spectrum}
$\cCSI(\varphi)$ is the map
$$
\cCSI\colon
\Z\stackrel{\eqref{eq:map-x_i}}{\longleftrightarrow}\bPP_1(H)
\longrightarrow\CSI(\varphi),
$$
where $\varphi=\varphi_H$, which is also a monotone bijection, and
$$
\CSI(\varphi)=\{\hmu(\bx)\mid \bx\in\bPP_1(H)\}
$$ 
is the \emph{mean index spectrum} of $H$ and the second arrow is the
map $\bx\mapsto \hmu(\bx)$.  Then \eqref{eq:act-index2} can be
rephrased as
$$
\cCS(H)=\frac{\lambda}{2(n+1)}\cCSI(\varphi)+\const,
$$
i.e., the action spectrum and the index spectrum agree up to a factor
and a shift. The factor can be made equal 1 by scaling $\omega$, and
the shift can be made zero by adding a constant to $H$. Then
\begin{equation}
\label{eq:equal-spectra}
\cCS(H)=\cCSI(\varphi).
\end{equation}

Finally, we have a different type of resonance relations involving
only the indices. To state the result, recall that for an un-capped
one-periodic orbit $x\in\PP_1(H)$ the mean index is well defined
modulo $2(n+1)=2N$, i.e., $\hmu(x)\in S_{2N}^1=\R/2(n+1)\Z$.  Let
$x_0,\ldots, x_n$ be the fixed points of a pseudo-rotation $\varphi_H$
of $\CP^n$. Then, as was shown in \cite{GK}, for some non-zero vector
$\vec{v}=(v_0,\ldots,v_n)\in \Z^{n+1}$
, we have
\begin{equation}
\label{eq:RR}
\sum v_i\hmu(x_i)= 0 \textrm{ in } \R/2(n+1)\Z.
\end{equation}
In other words, the closed subgroup $\Gamma\subset \T^{n+1}$
topologically generated by the \emph{mean index vector}
\begin{equation}
\label{eq:vDelta}
\vDelta=\vDelta(\varphi_H)
:=\big(\hmu(x_0),\ldots,\hmu(x_n)\big)\in
\T^{n+1}=\R^{n+1}/2(n+1)\Z^{n+1}
\end{equation}
has positive codimension. Moreover, the codimension is equal to the
number of linearly independent resonances, i.e., the rank of the
subgroup $\CR\subset\Z^{n+1}$ formed by all resonances $\vec{v}$; see
\cite{GK}.

Clearly, $\CR$ depends on $\varphi_H$. However, conjecturally, the
resonance relation
\begin{equation}
\label{eq:sum-RR}
\sum\hmu(x_i)=0 \textrm{ in } \R/2(n+1)\Z
\end{equation}
is universal, i.e., satisfied for all pseudo-rotations. (Up to a
factor this is the only possible universal resonance relation; for,
any other relation breaks down for a suitably chosen rotation; see
Section \ref{sec:R}.) For $\CP^1$ this conjecture is known to hold;
see the discussion in Section \ref{sec:2D} below. Theorem
\ref{thm:balanced-4D-intro} establishes this conjecture for strongly
non-degenerate pseudo-rotations of $\CP^2$.

\subsection{Pseudo-rotations of $S^2$}
\label{sec:2D}
To illustrate our approach, we will use now the results quoted in
Section \ref{sec:spectra} to study pseudo-rotations in dimension two.

\begin{Proposition}
\label{prop:S^2-nondeg}
Every pseudo-rotation $\varphi$ of $S^2$ is strongly non-degenerate
and its fixed points are elliptic. Furthermore, in the notation from
Section \ref{sec:spectra},
\begin{equation}
\label{eq:sum-RR-2D}
\hmu(\bx_0)+\hmu(\bx_1)=0.
\end{equation}
\end{Proposition}

The resonance relation \eqref{eq:sum-RR-2D}, which holds in $\R$ and
not just modulo 4, asserts, roughly speaking, that $D\varphi$ rotates
the tangent spaces at the fixed points by the same angle but in
opposite directions. In particular, \eqref{eq:sum-RR} is satisfied and
there exists a unique (up to conjugation) true rotation $R_\varphi$ of
$S^2$ such that $\cCS(\varphi)=\cCS(R_\varphi)$ and
$\cCSI(\varphi)=\cCSI(R_\varphi)$. The first part of the proposition
is a standard, although ultimately highly non-trivial, result in
two-dimensional dynamics (see \cite{Fr92}) and \eqref{eq:sum-RR-2D}
readily follows from the Poincar\'e--Birkhoff theorem, \cite[Appendix
A.2]{Br:Annals}; see also \cite{CKRTZ} for a different approach based
on \eqref{eq:RR}. The proof of this part given below is taken from
\cite{GG:PR}.

\begin{proof} Arguing by contradiction, assume that $\varphi$ is a
  pseudo-rotation of $S^2$ and some iterate of $\varphi$ is degenerate
  or that one of its fixed points is hyperbolic. In dimension two, a
  hyperbolic or degenerate fixed point necessarily has integral mean
  index. Hence, replacing $\varphi$ by a sufficiently large iterate if
  necessary and using \eqref{eq:RR}, we may assume that both fixed
  points $x_0$ and $x_1$ have mean index equal to zero modulo
  $4=2(n+1)$. Let us scale the symplectic structure and adjust the
  Hamiltonian so that $\cCS(H)=\cCSI(\varphi)$. Then, by
  \eqref{eq:act-index2}, for suitable cappings of $x_0$ and $x_1$
  these orbits have equal actions, which is impossible by Theorem
  \ref{thm:bijection}.

  Let us now turn to \eqref{eq:sum-RR-2D}. By Theorem
  \ref{thm:bijection}, for any iteration $k$ the orbits $x_0^k$ and
  $x_1^k$ with any cappings have distinct Conley--Zehnder indices. In
  particular,
  \begin{equation}
    \label{eq:change}
\mu\big(\bx_1^k\big)-\mu\big(\bx_0^k\big)\equiv 2\mod 4=2(n+1).
\end{equation}
Note also that, by \eqref{eq:mean-cz}, $-2<\hmu(\bx_0)<0$ and
$2>\hmu(\bx_1)>0$, since $\mu(\bx_0)=-1$ and $\mu(\bx_1)=1$ by
\eqref{eq:index-bxi}.

Therefore, again by \eqref{eq:mean-cz},
$\mu\big(\bx_1^{k+1}\big)-\mu\big(\bx_1^k\big)$ is either 0 or 2 and
$\mu(\bx_0^{k+1})-\mu(\bx_0^k)$ is either 0 or -2. Combining these
facts it is not hard to see that
$$
\mu\big(\bx_1^{k+1}\big)-\mu\big(\bx_1^k\big)=
-\big(\mu\big(\bx_0^{k+1}\big)-\mu\big(\bx_0^k\big)\big)
$$
for all $k\in\N$. (Otherwise, at the first moment when one index
changes and the other does not, \eqref{eq:change} is violated.) As a
consequence,
$$
\mu\big(\bx_1^k\big)=-\mu\big(\bx_0^k\big)
$$
for all $k\in\N$. Since
$\hmu(\bx_i)=\lim_{k\to\infty}\mu\big(\bx_i^k\big)/k$, we conclude
that $\hmu(\bx_1)=-\hmu(\bx_0)$.
  \end{proof}

\section{Index divisibility}
\label{sec:index}
In this section we state the main index theory result used in our
comparison of true rotations and pseudo-rotations. Let
$\Phi\in\tSp(2m)$ be a path parametrized by $[0,1]$, starting at the
identity and taken up to homotopy with fixed endpoints. Denote by
$\sigma(\Phi)$ the unit spectrum (the collection of unit eigenvalues)
of $\Phi(1)$. In general, \emph{we will refer to the eigenvalues of
  $\Phi(1)$ as the eigenvalues of $\Phi$ unless specifically stated
  otherwise}. Throughout this section we assume that $\Phi$ is
strongly non-degenerate, i.e., all iterates $\Phi^k$ are
non-degenerate or, equivalently, that $\sigma(\Phi)$ does not contain
any root of unity. We denote the part of $\sigma(\Phi)$ lying in the
upper half circle by $\sigma_+(\Phi)$ and refer to it as the positive
unit spectrum.

Our goal is to associate to $\Phi$ several numerical invariants which
determine its index behavior under iterations. For the sake of
simplicity let us first assume that all unit eigenvalues of $\Phi$ are
semi-simple although not necessarily distinct. Let us write $\Phi$ as
a product of a loop $\phi$ with the direct sum
$\Phi_{h}\oplus \Phi_{-h}\oplus \Phi_e$. Here $\Phi_h$ is hyperbolic
with complex or positive real eigenvalues and zero mean index. (We
emphasize that by our convention these are in fact the eigenvalues of
$\Phi_h(1)$.) An equivalent condition is that $\Phi_h(t)$ is
hyperbolic with complex or positive real eigenvalues for all
$t\in (0,\,1]$.  The second term $\Phi_{-h}$ is hyperbolic with
negative real eigenvalues. As an element of the universal covering it
is specified by that its mean index is equal to half of the dimension
of its domain or equivalently by that it is connected to $I$ by the
counterclockwise rotation by $\pi$ and a hyperbolic transformation. We
set
\[
\mult_{-1}(\Phi):=\hmu(\Phi_{-h}) \qquad and \qquad
\loop(\Phi):=\hmu(\phi).
\]

The remaining term $\Phi_e$ is elliptic. Up to conjugation, it
decomposes as a direct sum of ``short'' rotations $R_{\theta}$ by an
angle $\pi\theta\in (-\pi,\pi)$. Thus, in particular,
$\sigma(\Phi)=\sigma(\Phi_e)$ is the collection of eigenvalues
$\exp\big(\pm\pi\sqrt{-1}\theta\big)$. For a rotation $R_{\theta}$ the
eigenvalue $\exp\big(\pi\sqrt{-1}\theta\big)$ is said to be of the
first (Krein) type. In other words, for a counter-clock-wise rotation
we pick up the eigenvalue in the upper half-plane as the first Krein
type, and the eigenvalue in the lower half-plane for a clock-wise
rotation. The first type eigenvalues are actually defined for linear
symplectic maps, i.e., elements of $\Sp(2n)$ rather than elements of
the universal covering $\tSp(2n)$. We will often refer to the first
Krein type eigenvalues of the linearized return map at a periodic
orbit as the \emph{Floquet multipliers} of that orbit.

It is not hard to see that the loop $\phi$ and the paths $\Phi_h$,
$\Phi_{-h}$ and $\Phi_e$ are uniquely determined by $\Phi$ as elements
of the universal covering of the group of linear symplectic
transformations.

The \emph{signature} $\sgn_{\lambda}(\Phi)$ of
$\lambda\in\sigma_+(\Phi)$ is by definition the difference $p-q$,
where $p$ is the number of times $\lambda$ enters $\sigma(\Phi)$ as
the first type eigenvalue and $q$ is the number of times $\bl$ occurs
as the first type eigenvalue. (This is indeed the Krein signature of
the corresponding complex eigenspaces of $\Phi$ or, to be more
precise, of $\Phi(1)$. We refer the reader to, e.g., \cite{Ab, SZ, Lo}
for the definition of the Krein signature.) It is convenient to define
$\sgn_\lambda(\Phi)$ for all $\lambda\in S^1$ by setting
$\sgn_{\lambda}(\Phi):=\sgn_{\bl}(\Phi)$ (no negative sign!)  when
$\lambda\in\sigma(\Phi)\setminus \sigma_+(\Phi)$ and
$\sgn_{\lambda}(\Phi)=0$ when
$\lambda \in S^1 \setminus \sigma(\Phi)$.  Thus we can think of
$\lambda\mapsto \sgn_\lambda(\Phi)$ as a function on $S^1$ which
depends on $\Phi$.

It is not hard to extend these definitions to transformations $\Phi$
with not necessarily semi-simple eigenvalues. To this end, we may, for
instance, connect $\Phi$ to a transformation $\Phi'$ with all unit
eigenvalues semi-simple and the same spectrum as $\Phi$ via a family
of isospectral tranformations, i.e., a family of tranformations with
constant spectrum. We set $\mult_{-1}(\Phi):=\mult_{-1}(\Phi')$ and
$\loop(\Phi):=\loop(\Phi')$ and
$\sgn_{\lambda}(\Phi):=\sgn_{\lambda}(\Phi')$. Then it is easy to show
that these invariants are independent of the choice of $\Phi'$.

\begin{Remark}
  \label{rmk:ss}
  In fact, in the analysis of index behavior under iterations one can
  always use this trick to assume that the map $\Phi(1)$ is
  semi-simple. The index sequences $\mu\big(\Phi^k\big)$ and
  $\mu\big((\Phi')^k\big)$, and likewise the sequences of the mean
  indices $\hmu\big(\Phi^k\big)$ and $\hmu\big((\Phi')^k\big)$, are
  identical.
\end{Remark}

It is clear that these invariants are additive with respect to direct
sum and that $\loop+\mult_{-1}$ and $\sgn_\lambda$ change sign when
the tranformation is replaced by its inverse.

Essentially by definition,
\begin{equation}
\label{eq:mult-index}
\mu(\Phi)=\loop(\Phi)+\mult_{-1}(\Phi)+
\sum_{\lambda\in\sigma_+(\Phi)}\sgn_{\lambda}(\Phi)
\end{equation}
and
\begin{equation}
\label{eq:mult-mean-index}
\hmu(\Phi)=\loop(\Phi)+\mult_{-1}(\Phi)+
\sum_{\lambda\in\sigma_+(\Phi)}\sgn_{\lambda}(\Phi)\theta,
\end{equation}
where $\lambda=\exp\big(\pi\sqrt{-1}\theta\big)$ with
$\theta\in (0,\,1)$.

Our key combinatorial result is the following theorem proved in
Section~\ref{sec:index-div-pf}.

\begin{Theorem}[Index Divisibility]
  \label{thm:index-div}
  Fix $l\in\N$. The following two conditions are equivalent:
  \begin{itemize}
  \item[\rm(a)]
    $2l\mid \big( \mu\big(\Phi^{k+1}\big)-\mu\big(\Phi^k\big) \big)$
    for all $k\in\N$;
  \item[\rm(b)] the following two divisibility requirements are met:
\begin{itemize}
\item[\rm(i)] $2l \mid \big( \loop(\Phi)+\mult_{-1}(\Phi) \big)$,
\item[\rm(ii)] $l \mid \sgn_{\lambda}(\Phi)$ for all
  $\lambda\in\sigma(\Phi)$.
\end{itemize}
\end{itemize}
\end{Theorem}

This is Theorem \ref{thm:index-div-intro} from the introduction. Note
that Condition (a) is satisfied whenever $2l\mid \mu\big(\Phi^k\big)$
for all $k\in \N$. The latter requirement is stronger than and not
equivalent to (a) or (b) as simple examples show. (By
\eqref{eq:mult-index} and \eqref{eq:difference}, what follows from (a)
or (b) is only that $l\mid \mu\big(\Phi^k\big)$.)  However, we can
infer from the theorem that the sequence of indices or the sequence of
index jumps determines the eigenvalues with signature, except for the
zero signature eigenvalues:

\begin{Corollary} 
\label{cor:index-spec}
Let $\Phi$ and $\Psi$ be strongly non-degenerate. Then the following
three conditions are equivalent:
\begin{itemize}
\item[\rm(a)] $\mu(\Phi^k)=\mu(\Psi^k)$ for all $k\in\N$;
\item[\rm(b)]
  $\mu\big(\Phi^{k+1}\big)-\mu\big(\Phi^k\big)
  =\mu\big(\Psi^{k+1}\big)-\mu\big(\Psi^k\big)$
  for all $k\in\N$;
\item[\rm(c)] $\sgn_\lambda(\Phi)=\sgn_\lambda(\Psi)$ for all
  $\lambda\in S^1$ and
$$
\loop(\Phi)+\mult_{-1}(\Phi)=\loop(\Psi)+\mult_{-1}(\Psi).
$$
\end{itemize}
\end{Corollary}

\begin{Remark}
  \label{rmk:determine}
  In general, neither the index sequence $\mu\big(\Phi^k\big)$ nor the
  jump sequence $\mu\big(\Phi^{k+1}\big)-\mu\big(\Phi^k\big)$
  completely determines $\sigma(\Phi)$. For instance, zero signature
  eigenvalues are not detected in Theorem \ref{thm:index-div} and
  Corollary \ref{cor:index-spec}. In Corollary
  \ref{cor:index-spec}(c), it is possible for $\lambda $ to be a zero
  signature eigenvalue of $\Phi$ but not an eigenvalue of $\Psi$.
  \end{Remark}

  \begin{proof}[Proof of Corollary \ref{cor:index-spec}]
  Clearly, (a) implies (b). Assume that (b) holds. Then
  $$
  \mu\big((\Phi\oplus\Psi^{-1})^{k+1}\big)-
  \mu\big((\Phi\oplus\Psi^{-1})^k\big)
  =0
  $$
  and thus, by Theorem \ref{thm:index-div} applied to this
  tranformation,
$$
\sgn_\lambda \big(\Phi\oplus\Psi^{-1}\big)
=0
$$
for all $\lambda \in S^1$ and 
$$
\loop\big(\Phi\oplus\Psi^{-1}\big)
+\mult_{-1}\big(\Phi\oplus\Psi^{-1}\big)=0.
$$
Here we are using the fact that 0 is the only integer divisible by
infinitely many integers. Now, by additivity, we see that the
signatures and $\loop+\mult_{-1}$ for $\Phi$ and $\Psi$ are equal.

To prove that (c) implies (a) we need to show that
$\loop(\Phi)+\mult_{-1}(\Phi)$ and the signatures for $\Phi$ determine
these invariants for all iterations $\Phi^k$. This is easy to prove
directly. Alternatively, by Theorem \ref{thm:index-div},
$\loop(\Phi)+\mult_{-1}(\Phi)$ and the signatures for $\Phi$ determine
the jump sequence $\mu\big(\Phi^{k+1}\big)-\mu\big(\Phi^k\big)$ and
also, by \eqref{eq:mult-index}, the initial condition
$\mu(\Phi)$. Hence they also determine the sequence
$\mu\big(\Phi^k\big)$.
\end{proof}

\section{Pseudo-rotations vs.\  rotations}
\label{sec:PRvsR}

\subsection{True rotations}
\label{sec:R}
Consider a true rotation of $\CP^n$, i.e., a Hamiltonian
diffeomorphism $\varphi_Q$ of $\CP^n$ generated by a quadratic
Hamiltonian
\begin{equation}
\label{eq:Q}
Q(z)=\sum_{i=0}^n a_i |z_i|^2,
\end{equation}
where we have identified $\CP^n$ with the quotient of the unit sphere
$S^{2n+1}\subset \C^{n+1}$ and renormalized the standard symplectic
form $\omega$ on $\CP^n$ so that
$$
\int_{\CP^1}\omega=1.
$$
(For the standard normalization this integral is $\pi$.)

Most of the time it will be convenient to order the eigenvalues $a_i$
of $Q$ so that
\begin{equation}
\label{eq:order_lambda}
a_0\leq \ldots\leq a_n.
\end{equation}
Furthermore, since the Hamiltonian $\sum|z_i|^2$ reduces to a constant
Hamiltonian on $\CP^n$, we can assume without loss of generality that
\begin{equation}
\label{eq:sum_lambda}
\sum a_i=0,
\end{equation}
which is equivalent to the condition that $Q$ is normalized, i.e.,
$$
\int_{\CP^n} Q\,\omega^n=0.
$$
Finally, $\varphi_Q$ is non-degenerate if and only if
$a_i-a_j\not\in\Z$ and strongly non-degenerate if and only if
$a_i-a_j\not\in\Q$. Among the periodic orbits of $\varphi_Q$ are the
coordinate axes $x_0,\ldots,x_n$, all of which are one-periodic, and
these are the only periodic orbits when $\varphi_Q$ is strongly
non-degenerate. Thus $\varphi_Q$ is then a pseudo-rotation. We will
assume this to be the case from now on unless specifically stated
otherwise.

The first type eigenvalues of $D\varphi_Q$ at $x_i$ are
$\exp\big(2\pi\sqrt{-1}(a_i-a_j)\big)$ where $j\neq i$. Viewing
$\varphi_Q$ as an element in the universal covering
$\widetilde{\mathit{Ham}}(\CP^n,\omega)$ generated by the flow of $Q$,
we have the mean index of $\varphi_Q$ at $x_i$ defined once $x_i$ is
equipped with a capping. The orbit $x_i$ is constant and, in
particular, it can be given a trivial capping. We denote the resulting
trivially capped orbit by $\mathring{x}_i$. It is easy to see that
\begin{equation}
\label{eq:Delta}
\hmu(\rx_i)=2\sum_{j\neq
  i}(a_i-a_j)=-2\sum_j a_j+2(n+1)a_i
=2(n+1)a_i,
\end{equation}
where in the last equality we used \eqref{eq:sum_lambda}.  With $a_i$
arranged in an increasing order as in \eqref{eq:order_lambda}, the
Conley--Zehnder index of $\rx_i$ is $-n+2i$ when $Q$ is small. Without
the latter requirement, $\mu(\rx_i)$ can be any integer of the same
parity as $n$. The action of $\rx_i$ is $a_i$, and hence we have
\begin{equation}
\label{eq:rot-spec}
\CS(Q)=\coprod_i( a_i+\Z),
\end{equation}
where as above we assumed $\varphi_Q$ to be non-degenerate. 

For our purposes, however, it is more useful to cap $x_i$ so that its
Conley--Zehnder index is in the range from $-n$ to $n$ even when $Q$
is large. There exists exactly one such capping of $x_i$ and, as in
Sections \ref{sec:spectra} and \ref{sec:R}, we denote the resulting
capped orbit by $\bx_i$. Then the indices $\mu(\bx_i)$ for
$i=0,\ldots,n$ comprise the (un-ordered) collection of integers
$-n,\, -n+2,\ldots, n-2,\, n$. This recapping also affects the mean
indices.

\begin{Lemma}
\label{lemma:new_sum}
We have 
$$
\sum \hmu(\bx_i)=0.
$$
\end{Lemma}
It might be worth emphasizing that this equality is in $\R$ and not
just modulo $2(n+1)$.

\begin{proof}
  The total recapping from the orbits $\rx_i$ to the orbits $\bx_i$ is
$$
\sum\hmu(\bx_i)-\sum\hmu(\rx_i)=\sum\mu(\bx_i)-\sum\mu(\rx_i),
$$
where we identified $\pi_2(\CP^n)$ with $2(n+1)\Z$ via
$2c_1(T\CP^n)$. By \eqref{eq:sum_lambda} and \eqref{eq:Delta}, we have
$$
\sum\hmu(\rx_i)=0,
$$
and hence 
$$
\sum\hmu(\bx_i)=\sum\mu(\bx_i)-\sum\mu(\rx_i).
$$
The first sum on the right-hand side is zero. For, after if necessary
rearranging the terms,
$$
\sum\mu(\bx_i)=-n+ (-n+2)+\cdots+(n-2)+n=0.
$$
We also have 
$$
\sum\mu(\rx_i)=0.
$$
To see this, note that this sum is the Conley--Zehnder index of the
path
$$
\Phi(t)=\bigoplus\, D\varphi^t|_{x_i}, \quad t\in [0,1],
$$
in $\Sp\big(2n(n+1)\big)$. The first type eigenvalues of $\Phi$ break
down into complex conjugate pairs
$\exp\big(2\pi\sqrt{-1}(a_i-a_j)\big)$ where $j\neq i$ and the
eigenvalues within each pair have the same multiplicity. Hence
$\mu(\Phi)=0$, which concludes the proof of the lemma.
\end{proof}

A rotation $\varphi_Q$, viewed as an element of $\tHam(\CP^n,\omega)$
and not necessarily non-degenerate, lies in
$\SU(n+1)=\widetilde{\PU}(n+1)$ and conversely every element of
$\SU(n+1)$ can be generated by a diagonal quadratic Hamiltonian $Q$ as
in \eqref{eq:Q} for a suitable choice of coordinates.  Let us next
examine the question when $\varphi_Q$ is trivial as an element of this
universal covering, i.e., $\varphi_Q=\id$ in $\tHam(\CP^n,\omega)$.

Clearly, $\varphi_Q$ is trivial if and only if it is a contractible
loop.  For this, first of all, the path $\varphi^t_Q$, $t\in [0,1]$,
must be a loop, which is equivalent to that $a_i-a_j\in\Z$ for all $i$
and $j$. Next, this loop must be contractible, i.e., it must represent
the zero class in $\pi_1\big(\Ham(\CP^n,\omega)\big)$.  Note that we
have the ``Maslov index'' homomorphism
$\pi_1\big(\Ham(\CP^n,\omega)\big)\to \Z_{n+1}=\Z/(n+1)\Z$ given by
the evaluation of one-half of the mean index of a loop modulo $n+1$ on
any orbit. In particular, for the loop $\varphi_Q^t$ we can evaluate
the mean index at any of the fixed points. The composition
$$
\Z_{n+1}
=\pi_1\big(\PU(n+1)\big)\to \pi_1\big(\Ham(\CP^n,\omega)\big)\to
\Z_{n+1}
$$
is an isomorphism. Hence the loop $\varphi_Q$, which actually lies in
$\PU(n+1)$, is trivial in $\tHam(\CP^n,\omega)$ if and only if it is
trivial in $\widetilde{\PU}(n+1)$ and if and only if $a_i-a_j\in\Z$
for all $i$ and $j$ and $(n+1)a_i\in \Z$ for one eigenvalue $a_i$ or
equivalently for all eigenvalues. (The argument goes back to
Weinstein, \cite{We}, and comprises a linear algebra counterpart of
the Seidel representation, \cite{Sei}). This observation, of course,
readily translates into a criterion in terms of the eigenvalues for
two simultaneously diagonalizable quadratic Hamiltonians to generate
the same rotation. For our purposes, however, the following condition
expressed via action spectra is more useful.

\begin{Lemma}
\label{lemma:rot-spectra}
Two non-degenerate rotations $\varphi_Q$ and $\varphi_{Q'}$ with
simultaneously diagonalizable $Q$ and $Q'$ are equal as elements of
$\widetilde{Ham}(\CP^n,\omega)$ if and only if they have the same
action spectrum:
$$
\CS(Q)=\CS(Q').
$$
\end{Lemma}

\begin{proof}
  The action spectrum is completely determined by an element of
  $\tHam(\CP^n)$ and hence we only need to show that two rotations
  with equal action spectra are the same.

  Let $a_i$ and $a'_i$ be the eigenvalues of $Q$ and, respectively,
  $Q'$ normalized to satisfy \eqref{eq:order_lambda} and
  \eqref{eq:sum_lambda}. By \eqref{eq:rot-spec}, we have
$$
 \coprod_i a_i+\Z=\coprod_i a'_i+\Z=:\CS.
$$
Our goal is to show that $a_i=a'_i$ for all $i$. The actions $a_i$ or
$a'_i$ for $i=0,\ldots, n$ are $n+1$ consequent points in $\CS$ with
sum equal to zero.  (This follows from the fact that the ordering of
$\CS$ by the Conley--Zehnder index agrees with the ordering of $\CS$
by the action, i.e., as a subset of $\R$; see Theorem
\ref{thm:bijection}.) There is at most one way to pick up such $n+1$
consequent points, and hence $a_i=a'_i$.
\end{proof}

\subsection{Pseudo-rotations}
\label{sec:PR}
Consider a non-degenerate pseudo-rotation $\varphi$ of $\CP^n$ and
let, as in Section \ref{sec:spectra}, $\bx_0,\ldots,\bx_n$ be its
fixed points uniquely capped so that $\big|\mu(\bx_i)\big|\leq
n$. Furthermore, it will often be convenient to order the fixed points
by requiring the Conley--Zehnder index (or equivalently the action) to
increase:
$$
\mu(\bx_0)=-n,\,\mu(\bx_1)=-n+2,\,\ldots,\, \mu(\bx_n)=n.
$$

\begin{Definition}
\label{def:balanced}
A non-degenerate pseudo-rotation $\varphi$ is \emph{balanced} if
\begin{equation}
\label{eq:balanced}
\sum_i\hmu(\bx_i)=0.
\end{equation}
\end{Definition}

\begin{Example}
  By Lemma \ref{lemma:new_sum}, every strongly non-degenerate (true)
  rotation is balanced.
\end{Example}

\begin{Remark}
  Replacing the Conley--Zehnder index by the LS-index (see Section
  \ref{sec:spectra}) we can extend this definition to all, not
  necessarily non-degenerate, pseudo-rotations.
\end{Remark}

Note that the condition that $\varphi$ is balanced does not
automatically imply that all iterates $\varphi^k$ are
balanced. However, as is easy to see, these iterates are balanced
modulo $2(n+1)$, i.e., \eqref{eq:sum-RR} which is a slightly weaker
condition holds.  Hypothetically, every pseudo-rotation is balanced at
least under suitable non-degeneracy conditions. Then \eqref{eq:sum-RR}
is also satisfied and thus \eqref{eq:sum-RR} would indeed be a
universal resonance relation; see Section \ref{sec:background}. By
Proposition \ref{prop:S^2-nondeg} or the Poincar\'e--Birkhoff theorem
(cf.\ \cite[Appendix A.2]{Br:Annals}), every pseudo-rotation of $S^2$
is balanced. Furthermore, this conjecture, which in a somewhat
different form is already stated in \cite{GK}, is supported by the
following result (Theorem \ref{thm:balanced-4D-intro} from the
introduction):

\begin{Theorem}
  \label{thm:balanced-4D}
  Let $\varphi$ be a strongly non-degenerate pseudo-rotation of
  $\CP^2$. Then all fixed points of $\varphi$ are elliptic and, unless
  it is a ghost pseudo-rotation in the sense of Definition
  \ref{def:ghost}, $\varphi$ is balanced.
\end{Theorem}

We discuss ghost pseudo-rotations in more detail in Section
\ref{sec:Ghosts}.

\begin{proof}
  Let $\varphi\colon \CP^2\to\CP^2$ be a strongly non-degenerate
  pseudo-rotation. As above, we denote its capped fixed points with
  Conley--Zehnder indices $-2$, $0$ and $2$ by, respectively, $\bx_0$,
  $\bx_1$ and $\bx_2$.

  To see that all three fixed points of $\varphi$ are necessarily
  elliptic, assume otherwise. Then, since the dimension is four, at
  least one of the points, say, $x$ is either hyperbolic or hyperbolic
  elliptic. In the latter case, $\mu(x)$ is odd if the hyperbolic part
  of $x$ is positive, and $\mu\big(x^2\big)$ is odd if the hyperbolic
  part of $x$ is negative. Since all periodic points of $\varphi$ have
  even index, this is impossible. Hence at least one fixed point of
  $\varphi$ is hyperbolic.  However, in this case, $\varphi$ must have
  infinitely many periodic orbits; see \cite {GG:hyperbolic} or
  \cite{GG:PR}. This shows that all fixed points of $\varphi$ are
  elliptic.

  The key to the proof of the theorem is the fact that, by Theorem
  \ref{thm:bijection}, for all iterations $k\in \N$ the orbits $x_i^k$
  have distinct even indices modulo $2(n+1)=6$. (Recall that the
  indices modulo $2(n+1)$ are well-defined without capping.)

  As in the proof of Lemma \ref{lemma:new_sum}, consider
$$
\Phi(t)=\bigoplus_i D\varphi^t|_{\bx_i}, \quad t\in [0,\,1],
$$
which we will treat as an element of
$\tSp\!\big(2n(n+1)\big)=\tSp(12)$. Note that here, as the notation
indicates, we use the capping of $\bx_i$ to turn $D\varphi^t|_{x_i}$
into an element of $\tSp(2n)=\tSp(4)$.  As in Remark \ref{rmk:ss}, we
can assume without loss of generality throughout the proof that the
linearized maps $D\varphi|_{x_i}$ and the sum $\Phi(1)$ are
semisimple. Clearly,
\begin{equation}
\label{eq:mu(Phi)}
\mu(\Phi)=\sum\mu(\bx_i)=0
\end{equation}
and 
$$
\mu\big(\Phi^k\big)=\sum\mu\big(\bx_i^k\big)=0\mod 2(n+1)=6
$$
due to the fact that for every $k$ the indices $\mu(\bx_i^k)$ assume
only the values $-2$, $0$ and $2$ modulo 6, and are all distinct.

Furthermore, $\mult_{-1}(\Phi)=0$ since the orbits $x_i$, and hence
the transformation $\Phi$, are elliptic. Applying Theorem
\ref{thm:index-div} to $\Phi$, we see that $6\mid \loop(\Phi)$ and
$3\mid\sgn_\l(\Phi)$ for all $\l\in\sigma_+(\Phi)$. Since
$\sigma_+(\Phi)$ comprises exactly six eigenvalues counting with
multiplicity, there are now three possibilities:
\begin{itemize}
\item[\emph{Case 1}:] There are three eigenvalues (not necessarily
  distinct) in $\sigma_+(\Phi)$ and all eigenvalues have multiplicity
  2 and signature 0.
\item[\emph{Case 2}:] All eigenvalues in $\sigma_+(\Phi)$ are equal
  and the signature is $\pm 6$.
\item[\emph{Case 3}:] There are two distinct eigenvalues in
  $\sigma_+(\Phi)$ and their signatures are $\pm 3$.
\end{itemize}

Note that these three cases are completely determined by
$\Phi(1)\in \Sp(12)$ rather than $\Phi\in \tSp(12)$. Case 1 includes
the situation where two of the eigenvalues or even all three are
equal, although the latter, as we will see, is impossible.  In any
event, in Case 1 we also have $\loop(\Phi)=0$ due to
\eqref{eq:mult-index} and \eqref{eq:mu(Phi)}. Therefore, by
\eqref{eq:mult-mean-index},
$$
\sum\hmu(\bx_i)=\hmu(\Phi)=\loop(\Phi)=0,
$$
and the proof is finished.

In the rest of the proof we will rule out Case 2 and, assuming that
$\varphi$ is not a ghost pseudo-rotation, Case 3. Before turning to
that, let us point out that $\Phi_{\varphi^k}\neq \Phi_{\varphi}^k$,
where we included the dependence of $\Phi$ on $\varphi$ in the
notation. However, $\Phi_{\varphi^k}(1)= \Phi_{\varphi}^k(1)$ or, in
other words, the two elements of $\tSp(12)$ only differ by a loop. In
particular, they have the same spectra and signatures. Hence, in what
follows, when only this information is concerned we can replace
$\Phi_{\varphi^k}$ by $\Phi_{\varphi}^k$.

Let us rule out Case 2. Then $\Phi$ is the direct sum of six equal
clockwise (when $\sgn_\lambda(\Phi)=-6$) or counter-clockwise (when
$\sgn_\lambda(\Phi)=6$) rotations and possibly some loop.  Observe
that this case is invariant under iteration, although the signature
value may change sign. Denote by $\lambda$ the only eigenvalue in
$\sigma_+(\Phi)$. The positive semi-orbit $\{\lambda^k\mid k\in\N\}$
contains $1\in S^1$ in its closure. (In fact, by non-degeneracy, this
orbit is dense in $S^1$.) Hence, by passing to an iterate of
$\varphi$, we may assume that $\lambda$ is arbitrarily close to
$1$. (It suffices to ensure that the distance in $S^1$ from $\lambda$
to 1 is less than $\pi/6$.)  Then, passing to the 6-th iteration of
$\varphi$ we can guarantee that $\loop\big(\bx_i^6\big)=0$ modulo 6
for all three capped fixed points of $\varphi$. Next, let us redenote
$\varphi^6$ by $\varphi$ and $\lambda^6$ by $\lambda$, etc. After
recapping and relabeling the original orbits $\bx_i$, we obtain three
capped one-periodic orbits of $\varphi$ (again denoted by $\bx_i$)
with zero loop part, and hence
$\mu(\bx_i)=\sgn_\lambda\big(D\varphi^t|_{\bx_i}\big)$. Moreover,
these indices are distinct and equal to $\pm 2$ and 0. Furthermore,
the sum of the signatures is still equal to
$\sgn_\lambda(\Phi)=\pm 6$, which is impossible -- all three indices
would then have to be either $+2$ or $-2$.

The goal of the rest of the proof is to deal with Case 3. Denote the
distinct eigenvalues of $\Phi(1)$ in $\sigma_+(\Phi)$ by $\l$ and
$\eta$. Then it is not hard to see that $\Phi$ is the direct sum of
three copies of a path $\psi_0$ in $\U(1)$ and three copies of a path
$\psi_1$ in $\U(1)$. The end point $\psi_1(1)$ is either $\lambda$ or
$\bar{\l}$ depending on whether $\sgn_{\l}(\Phi)$ is positive and thus
equal to 3 or negative and equal to $-3$. Likewise, the end point of
$\psi_0$ is either $\eta$ or $\bar{\eta}$ depending on
$\sgn_{\eta}(\Phi)$.

Let $G:=G(\Phi)$ be the subgroup in $\T^2=S^1\times S^1$ topologically
generated by $(\lambda,\eta)$ and let $G_0$ be the connected component
of the identity in $G$. We necessarily have $\dim G=1$; for, by
\cite{GK}, these eigenvalues must satisfy a resonance relation of the
form $\lambda^m\eta^l=1$ for some integers $m$ and $l$, which are not
both equal to zero, and also since $\dim G\geq 1$ by
non-degeneracy. Furthermore, note that $G/G_0$ is cyclic since $G$ is
monothetic.

The group $G$ depends on $\Phi$ and it is essential for what follows
to understand how it changes under iterations of $\Phi$ (or
equivalently $\varphi$). To avoid ambiguity in labeling the
eigenvalues, we will fix a decomposition of $\Phi$ into the sum of the
three copies of $\psi_0$ and $\psi_1$ and use it to order the
eigenvalues of $\Phi^k$. Let $\lambda_k$ and $\eta_k$ be the
eigenvalues of $\psi_1(1)^k$ and, respectively, $\psi_0(1)^k$ in the
open upper half-circle $S^1_+\subset S^1$. Set
$\Pi=S^1_+\times S^1_+\subset \T^2$. Since $(\lambda, \eta)\in \Pi$,
we necessarily have $G\cap\Pi\neq \emptyset$. For instance, $G$ cannot
be the anti-diagonal $\bar{\Delta}$ in $\T^2$.

When $\Phi$ is replaced by $\Phi^k$, the pair $(\lambda,\eta)$ gets
replaced by $(\lambda_k,\eta_k)$ which is equal to
$(\lambda^k,\eta^k)$, but only up to complex conjugation. Indeed, if
$\Imm\lambda^k<0$, we need to replace $\lambda^k$ by $\bar{\l}^k$ and
similarly for $\eta$. The reason is that by construction we take the
eigenvalues in $S^1_+$. We call this instance a \emph{flip}. In other
words, the generator $(\lambda_k,\eta_k)$ of $G(\Phi^k)$ is obtained
from $(\lambda^k, \eta^k)$ by applying, if necessary, complex
conjugations to its coordinates to bring it into $\Pi$.

It follows that, in the obvious notation, $G_0\big(\Phi^k\big)$ is
either $G_0(\Phi)$ or is obtained from it by conjugating one of the
coordinates. Also, when passing from $\Phi$ to $\Phi^k$, the quotient
$G/G_0$ can become smaller: $G\big(\Phi^k\big)/G_0\big(\Phi^k\big)$ is
only a subgroup of $G(\Phi)/G_0(\Phi)$. For instance, one can find a
sequence of iterations $k\to\infty$ such that this quotient is
trivial. Indeed, $G$ contains points $(\lambda^k,\eta^k)$ arbitrarily
close to the unit $(1,1)$, and then
$(\l_k,\eta_k)\in G_0\big(\Phi^k\big)$ topologically generates
$G_0(\Phi^k)=G\big(\Phi^k\big)$.

Flips also determine the behavior of the signature under iterations
and it is not hard to see that
\begin{eqnarray}
  \sgn_{\l_k}\big(\Phi^k\big)& = & \sgn_\l(\Phi) \textrm{ if }\l^k\in S^1_+, \textrm{
                                   and } \label{eq:sgn-iter1}\\
  \sgn_{\l_k}\big(\Phi^k\big) &= &-\sgn_\l(\Phi) \textrm{ if } \bar{\l}^k\in S^1_+, \label{eq:sgn-iter2}
\end{eqnarray}
and likewise for $\eta$.

Before proceeding with the proof, we need to clear one terminological
point. Namely, in contrast with Cases 1 and 2, Case 3 is not closed
under iterations: it might happen that $\lambda_k=\eta_k$ while
$\lambda\neq \eta$. This is exactly the situation where $G_0$ is the
diagonal $\Delta$ or the anti-diagonal $\bar{\Delta}$. By passing to
an iterate of $\varphi$, we can make $G=\Delta$ or, equivalently,
$\sigma_+(\Phi)$ comprises exactly one point, denoted by $\lambda$,
either reducing Case 3 to Case 2 if $\lambda$ has signature $\pm 6$,
which we have proved to be impossible, or to a subcase of Case 1.

This is the subcase where $\sigma_+(\Phi)=\{\lambda\}$ and $\lambda$
has multiplicity six and zero signature, which, as we have already
mentioned, is also impossible. Indeed, by passing to an iterate, we
may ensure that as in Case 2 the loop parts of all fixed points are
zero modulo 6 and indices are $\pm 2$ and $0$. We can still assume
that the only eigenvalue in $\sigma_+(\Phi)$ is close to $1\in
S^1$. Then, after sufficiently many iterations, once the eigenvalue
makes its first ``jump'' the indices $\pm 2$ will change to $\pm 6=0$
modulo 6 which contradicts the fact that distinct periodic orbits must
have distinct indices modulo 6.

To summarize, we have shown that $G_0$ is not $\Delta$ or
$\bar{\Delta}$, and $\lambda_k\neq \eta_k$ for all $k\in \N$.

Since $G_0$ is connected, it can be parametrized as
\begin{equation}
  \label{eq:paramG}
\gamma(t)=\big(e^{2\pi\sqrt{-1} p t}, e^{2\pi\sqrt{-1} q t}\big)
\end{equation}
with $t\in [0,\, 1)$. Here, $p$ and $q$ are relatively prime integers,
and $p\neq 0$ and $q\neq 0$ due to non-degeneracy. (Furthermore, we
cannot have $p=1=q$ or $p=\pm q$.) We call $G_0$ a \emph{first
  quadrant subgroup} if $p>0$ and $q>0$.

We claim that
\begin{eqnarray*}
  \sgn_{\l}(\Phi) &=& -\sgn_\eta(\Phi)\textrm{ when } G_0 \textrm{ is
                      first quadrant, and }\\
  \sgn_{\l}(\Phi) &=& \sgn_\eta(\Phi) \textrm{ otherwise}.
\end{eqnarray*}
Let us show the former. By passing to an iterate and, by
\eqref{eq:sgn-iter1}, without changing the signatures we can reduce
the question to the case where $G$ is a connected first quadrant
subgroup and $(\lambda,\eta)$ is arbitrarily close to the identity. As
usual denote by $\bx_i$ the capped one-periodic orbits of $\varphi$
with indices $\pm 2$ and 0. Since $\l$ and $\eta$ are close to 1,
without changing the signatures of $\l$ and $\eta$ we can ensure that
$\loop(\bx_i)=0$ for all three fixed points by passing to
$\varphi^6$. Now if $\l$ and $\eta$ have the same signature, either 3
or $-3$, we arrive at a contradiction exactly as in Case 2. Namely,
then all three points $x_i^6$ have the same index modulo 6. (The index
is 2 if the signature is positive and $-2$ if the signature is
negative.) The case where $G_0$ is not a first quadrant subgroup is
treated similarly, but now, by \eqref{eq:sgn-iter2}, the signature of
exactly one of the eigenvalues changes sign when we bring $\l$ and
$\eta$ close to 1 and make $G$ connected.

Next, let us examine more closely the case where $G_0$ is a first
quadrant subgroup. Then the signature of one of the eigenvalues, say
$\l$, is 3 and the other, $\eta$, has signature $-3$. Consider the
intersection $G_0\cap\bar{\Delta}\subset
G\cap\bar{\Delta}$. Parametrizing $G_0$ as in \eqref{eq:paramG} we see
that the first of these is a subgroup of $\Delta$ of order $p+q$,
where $p\geq 1$ and $q\geq 1$ and $\gcd(p,q)=1$. Thus
$$
|G_0\cap\bar{\Delta}|=p+q\geq 3
$$
and, as a consequence,
$$
|G\cap\bar{\Delta}|\geq 3.
$$
When $|G\cap\bar{\Delta}|=3$, by definition $\varphi$ is a ghost
pseudo-rotation by Definition \ref{def:ghost}. Indeed, in this case,
the first type eigenvalues are $\lambda_1=\lambda$ and
$\eta_1=\bar\eta$ and $\lambda_1=\bar{\eta}_1^2$, after if necessary
swapping $\eta$ and $\lambda$. Alternatively, the first type
eigenvalues are $\lambda_1=\bar{\lambda}$ and $\eta_1=\eta$ and again
$\lambda_1=\bar{\eta}_1^2$. (We also note that $G$ is connected, i.e.,
$G=G_0$, when $|G\cap \Delta|=3$.) Hence, we need to rule out the case
$|G\cap\bar{\Delta}|\geq 4$.

The group $G\cap\bar{\Delta}\subset \bar{\Delta}\cong S^1$ is cyclic,
as any subgroup of $S^1$. Let $(\zeta,\bar{\zeta})$ be the generator
of this group uniquely determined by the conditions that it is closest
to $(1,1)$ in $\bar{\Delta}$ and such that $\Imm\zeta>0$, and hence
$\Imm \bar{\zeta}<0$. Thus $\zeta=e^{2\pi\sqrt{-1}/\ell}$, where
$\ell=|G\cap \bar{\Delta}|\geq 4$.

The positive semi-orbit of any topological generator is dense in
$G$. Therefore, there exists a sequence of iterations $k$ bringing
$(\lambda^k,\eta^k)$ arbitrarily close to $(\zeta,\bar{\zeta})$ in
$\T^2$. For such an iteration, $\eta$ flips but $\l$ does
not. Therefore, $\eta_k=\bar{\eta}^k$ and $\l_k=\l^k$ are both close
to $\zeta$ and, by \eqref{eq:sgn-iter1} and \eqref{eq:sgn-iter2},
$$
\sgn_{\l_k}\big(\Phi^k\big)=\sgn_\l(\Phi)=3
$$
and
$$
\sgn_{\eta_k}\big(\Phi^k\big)=-\sgn_\eta(\Phi)=3.
$$
In other words, the first Krein type eigenvalues of $\Phi^k(1)$ are
exactly $\l_k$ and $\eta_k$, each occurring with multiplicity 3. Both
of the eigenvalues are close to $\zeta$.

Denote the three capped one-periodic orbits of $\varphi^k$ of indices
$\pm 2$ and 0 by $\bx_i$ with $i=0,1,2$. Then $\Phi^k(1)$ is the
direct sum of the return maps $D\varphi^k|_{x_i}$ and all three return
maps must have only first type eigenvalues, which are then
$(\lambda_k,\lambda_k)$, $(\lambda_k,\eta_k)$ and
$(\eta_k,\eta_k)$. We do not know how these pairs of eigenvalues are
assigned to the orbits $x_i$, but, for all $i$,
$$
\mu(\bx_i)=2+\loop(\bx_i),
$$
where $\loop(\bx_i)$ is even.

Let us iterate $\varphi^k$ three more times. Then, since $\ell>3$ and
$\l_k$ and $\eta_k$ are close to $\zeta$, none of the eigenvalues
``jumps'' over 1, and
$$
\mu\big(\bx_i^3\big)=2+3\loop(\bx_i)
$$
for all $i$. It follows that $\mu\big(\bx_i^3\big)=2$ modulo $6$ for
all $i$ which is impossible.

When $G_0$ is not a fist quadrant subgroup. The argument is similar,
but we intersect $G$ with the diagonal $\Delta$. The integers $p$ and
$q$ from the parametrization \eqref{eq:paramG} of $G_0$ have now
opposite signs, but are still relatively prime, and $|p|\geq 1$ and
$|q|\geq 1$. Then
$$
|G\cap\Delta|\geq |G_0\cap\Delta|=|p-q|\geq 3.
$$
When $|G\cap \Delta|=3$, the map $\varphi$ is again a ghost
pseudo-rotation by Definition \ref{def:ghost}. Indeed, in this case
both $\lambda$ and $\eta$ are in $S^1_+$ and are first type
eigenvalues, and $\lambda=\bar{\eta}^2$, where we labeled by $\lambda$
the eigenvalue closest to $1\in S^1$. Alternatively, $\lambda$ and
$\eta$ are in $S^1_+$ but $\bar{\lambda}$ and $\bar{\eta}$ are first
type eigenvalues, and again $\lambda=\bar{\eta}^2$. (As before, $G$ is
connected.)  The rest of the argument is exactly as above with the
only difference that now no eigenvalue flips.
\end{proof}

One important feature of balanced pseudo-rotations $\varphi=\varphi_H$
is that the mean index spectrum $\CSI(\varphi)$ (or equivalently the
action spectrum) completely determines the marked spectrum
$\cCSI(\varphi)$; see Section \ref{sec:spectra}. This is an immediate
consequence of the observation that $\CSI(\varphi)$ contains at most
one collection of $n+1$ consecutive points $a_0,\ldots,a_n$ with
$\sum a_i=0$. These points are then assigned indices
$-n, -n+2,\ldots, n$ and the rest of the spectrum is labeled
accordingly.

Let $\varphi=\varphi_H$ be a balanced strongly non-degenerate
pseudo-rotation of $\CP^n$ and let, as above, $\bx_0,\ldots,\bx_n$ be
its fixed points capped so that $\mu(\bx_i)=2i-n$. Then there exists a
unique true rotation $R_\varphi$, called the \emph{matching rotation},
such that
$$
\CSI(R_\varphi)=\CSI(\varphi)
$$
or, equivalently,
$$
\CS(Q)=\CS(H)
$$
up to a shift, where $Q$ is the quadratic form generating $R_\varphi$.
Then, since for both $R_\varphi$ and $\varphi$ the index spectrum
determines the marked spectrum,
$$
\cCSI(R_\varphi)=\cCSI(\varphi).
$$
The rotation $R_\varphi$ can be given, for instance, by the
Hamiltonian
$$
Q(z)= \frac{1}{2(n+1)} \sum \hmu(\bx_i)|z_i|^2.
$$
The uniqueness follows from Lemma \ref{lemma:rot-spectra}. Denoting by
$\by_i$ the capped fixed points of $R_\varphi$ with cappings again
chosen so that $\mu(\by_i)=2i-n$, for all $i=0,\ldots,n$ we have
\begin{equation}
  \label{eq:equal-ind}
\hmu(\by_i)=\hmu(\bx_i).
\end{equation}

Our next two results show that in many instances the points $x_i$ and
$y_i$ have, roughly speaking, the same Floquet multipliers. To state
the results, it is convenient to introduce the notion of the
\emph{decorated spectrum} $\hsi(\Phi)$ of an element
$\Phi\in\tSp(2n)$. Namely, this is the collection of pairs
$\big(\lambda,\sgn_\l(\Phi)\big)\in S^1\times \Z$, where
$\l\in\sigma_+(\Phi)$ and $\sgn_\l(\Phi)\neq 0$. (Thus the eigenvalues
with zero signature do not register in the decorated spectrum, and
Corollary \ref{cor:index-spec} can be rephrased as that the sequence
of indices $\mu(\Phi^k)$ determines $\hsi(\Phi)$.)  For a capped
one-periodic orbit $\bx$, we set
$\hsi(\bx):=\hsi\big(D\varphi^t|_{\bx}\big)$, where as always we used
the capping of $\bx$ to turn the linearized flow into an element of
$\tSp(2n)$. Likewise, $\loop(\bx):=\loop\big(D\varphi^t|_{\bx}\big)$,
etc.

\begin{Theorem}
\label{thm:matching1}
Let $\varphi$ be a strongly non-degenerate pseudo-rotation of $\CP^n$
such that all pseudo-rotations $\varphi^k$, $k\in\N$, are balanced and
let $R_\varphi$ be its matching rotation. Then, in the above notation,
for all $i=0,\ldots,n$
\begin{equation}
\label{eq:spec-matching}
\hsi(\bx_i)=\hsi(\by_i)\textrm{ and } \loop(\bx_i)+\mult_{-1}(\bx_i)
=\loop(\by_i).
\end{equation}

\end{Theorem}

This theorem together with Theorem \ref{thm:balanced-4D} yields

\begin{Corollary}
\label{cor:matching-4D}
Let $\varphi$ be a strongly non-degenerate pseudo-rotation of $\CP^2$,
which is not a ghost pseudo-rotation, and let $R_\varphi$ be its
matching rotation. Then \eqref{eq:spec-matching} holds.
\end{Corollary}

\begin{proof}[Proof of Theorem \ref{thm:matching1}]
First, observe that 
\begin{equation}
  \label{eq:CSI0}
\CSI(\varphi^k)=\CSI(R^k_\varphi)
\end{equation}
since the left-hand side is determined by $\CSI(\varphi)$ and the
right-hand side by $\CSI(R_\varphi)$. Indeed, by \eqref{eq:equal-ind},
$$
\hmu\big(\by_i^k\big)=k\hmu(\by_i)=k\hmu(\bx_i)=\hmu\big(\bx_i^k\big)
$$
and the rest of the spectra is obtained by recapping. (Note however
that this does not automatically imply that $\varphi^k$ is balanced
but only that it is ``balanced modulo $2(n+1)$''.) Combining this with
the condition that the iterates $\varphi^k$ are balanced and the
observation mentioned above that for a balanced pseudo-rotation the
index or action spectrum determines the marked spectrum, we see that
\begin{equation}
\label{eq:CSI}
\cCSI(\varphi^k)=\cCSI(R^k_\varphi).
\end{equation}
As a consequence, recalling that
$\hmu\big(\by_i^k\big)=\hmu\big(\bx_i^k\big)$, we have
$\mu(\bx_i^k)=\mu(\by_i^k)$ for all $k\in\N$, and
\eqref{eq:spec-matching} now follows from Corollary
\ref{cor:index-spec}.
\end{proof}

The requirement that all iterations $\varphi^k$ are balanced can be
eliminated if we assume that all Floquet multipliers are
distinct. Namely, we have the following result.

\begin{Theorem}
\label{thm:matching2}
Let $\varphi$ be a balanced strongly non-degenerate pseudo-rotation
and let $R_\varphi$ be its matching rotation. Assume that for every
one-periodic orbit $y_i$ of $R_\varphi$ all unit eigenvalues at $y_i$
(i.e., the elements of $\sigma(y_i)$) are distinct and
$\sigma(y_i)\cap\sigma(y_j)=\emptyset$ for any pair $i\neq j$. Then
\eqref{eq:spec-matching} holds and, in particular,
$\sigma(x_i)=\sigma(y_i)$.  Furthermore, the fixed points of $\varphi$
are elliptic and all iterates $\varphi^k$ are balanced.
\end{Theorem}

This implies Theorem \ref{thm:matching2-intro} from the introduction.

\begin{Remark}
  The spectra $\sigma(y_i)$ are determined by the collection of the
  mean indices $\hmu(\by_j)=\hmu(\bx_j)$; see Section
  \ref{sec:R}. Namely, the assumption of the theorem can be explicitly
  rephrased as that all eigenvalues
  $\exp\big(2\pi\sqrt{-1}(\hmu(\bx_i)-\hmu(\bx_j))/2(n+1)\big)$, where
  $i\neq j$, are distinct.
\end{Remark}

\begin{proof}[Proof of Theorem \ref{thm:matching2}]
  As in the proof of Theorem \ref{thm:matching1}, \eqref{eq:CSI0}
  holds. However, we do not know if the equality extends to the marked
  spectra, i.e., if \eqref{eq:CSI} holds, for $k\geq 2$. The markings
  of the spectra by indices may differ by a shift of degree
$$
s(k)=\mu(\bx_i^k)-\mu(\by_i^k).
$$
As the notation suggests, $s(k)$ is independent of $i$.  For,
otherwise, the natural ordering of $\CSI(\varphi^k)\subset \R$ would
not agree with its ordering by the Conley--Zehnder index; see Section
\ref{sec:spectra}.  Furthermore, $\varphi^k$ is balanced if and only
if $s(k)=0$, for $R_\varphi^k$ is balanced; cf.\ \eqref{eq:CSI}.

Thus $s(1)=0$ since both $\varphi$ and $R_\varphi$ are
balanced. Denote by $\Phi_i$ the linearized flow at $\bx_i$ and by
$\Psi_i$ the linearized flow at $\by_i$. Then
$$
\mu\big(\big(\Phi_i\oplus\Psi_i^{-1}\big)^k\big)=s(k)
=\mu\big(\big(\Phi_j\oplus\Psi_j^{-1}\big)^k\big)
$$
for any $i$ and $j$, and, by Corollary \ref{cor:index-spec}, we have
\begin{equation}
\label{eq:spec-equal}
\hsi\big(\Phi_i\oplus\Psi_i^{-1}\big)
=\hsi\big(\Phi_j\oplus\Psi_j^{-1}\big).
\end{equation}
The conditions of the theorem translate as that the elements in
$\sigma_+(\Psi_i)=\sigma_+(\Phi_i)$ are distinct for all $i$ (and
hence have signature $\pm 1$) and that for any pair $i\neq j$
  \begin{equation}
    \label{eq:distinct-again}
    \sigma_+(\Psi_i)\cap\sigma_+(\Psi_j)=\emptyset.
  \end{equation}

  To prove the theorem, by Theorem \ref{thm:matching1} it is enough to
  show that the decorated spectra \eqref{eq:spec-equal} are empty,
  i.e., $\sgn_\lambda\big(\Phi_i\oplus\Psi_i^{-1}\big)=0$ for all
  $\lambda\in\sigma_+\big(\Phi_i\oplus\Psi_i^{-1}\big)$, and that
  $\loop\big(\Phi_i\oplus\Psi_i^{-1}\big)=0$ and the iterates
  $\varphi^k$ are balanced for all $k$. (Since the fixed points $y_i$
  are elliptic and all eigenvalues are distinct, this will imply that
  $x_i$ are also elliptic and $\mult_{-1}(x_i)=0$.)

  Let us first prove that the decorated spectra \eqref{eq:spec-equal}
  are empty. By \eqref{eq:spec-equal}, we only need to do this for one
  of them. For the sake of brevity, set
  $\Gamma_i=\hsi\big(\Phi_i\oplus\Psi_i^{-1}\big)$ and denote by
  $\pi\colon S^1\times \Z\to S^1$ the projection to the first
  coordinate. Then $\pi|_{\Gamma_i}$ is a one-to-one map and
  $\pi(\Gamma_i)\subset \sigma_+\big(\Phi_j\oplus\Psi_j^{-1}\big)$ is
  the set of eigenvalues with non-zero signature.

  Let $A_i\subset \Gamma_i$ be the subset comprising all pairs
  $(a,m)\in\Gamma_i$, where $a\in\sigma_+\big(\Psi_i^{-1}\big)$ but
  $a\not\in\sigma_+(\Phi_i)$. Furthermore, we automatically have
  $m=\sgn_a\big(\Psi_i^{-1}\big)\neq 0$ since $(a,m)\in A_i$.  Our
  first goal is to show that at least two of the sets $A_i$ are empty.

  To this end, denote by $B_i$ the complement to $A_i$ in
  $\Gamma_i$. The set $B_i$ consists of pairs $(a,m)\in\Gamma_i$ with
  $a\in\sigma_+(\Phi_i)$. In other words, $(a,m)\in B_i$ if and only
  if $a\in \sigma_+(\Phi_i)$ and
  $m=\sgn_a(\Phi_i)-\sgn_a(\Psi_i)\neq 0$.  Observe that $|B_i|\leq n$
  and $|A_i|\leq n$ and $\pi(A_i)\cap \pi(A_j)=\emptyset$ for
  $i\neq j$ by \eqref{eq:distinct-again}, and, as a consequence, the
  sets $A_i$ are disjoint. (Henceforth, $|\cdot |$ stands for the
  cardinality.)

  We claim that in fact $|B_i|<n$ for all $i$. Without loss of
  generality we may set $i=0$ and note that
  $|B_0|\leq |\sigma_+(\Phi_0)|\leq n$. It suffices to show that
  $|B_0|\neq n$. Arguing by contradiction assume the contrary:
  $|B_0|=n$. Then $\sgn_a(\Phi_i)-\sgn_a(\Psi_i)\neq 0$ for every
  $a\in \sigma_+(\Phi_0)$. Thus $\sigma_+(\Psi_0)$ is contained in
  $\pi(\Gamma_0)$ -- none of its points gets cancelled by
  $\sigma_+(\Phi_0)$. Therefore, we have a one-to-one map
  $\hsi(\Psi_0)\cong\sigma_+(\Psi_0)\to \Gamma_j$ sending a point $a$
  to the unique pair $(a,m)\in
  \hsi\big(\Phi_j\oplus\Psi_j^{-1}\big)$. (This map need not preserve
  the signature and its image may overlap with $B_0$.) Pick $l$ such
  that $0<l\leq n$. By \eqref{eq:spec-equal} and
  \eqref{eq:distinct-again}, the image of this map is contained in
  $B_l$ and, since $|B_l|\leq n$, we have $|B_l|=|\hsi(\Psi_0)|=n$ and
  $\pi(B_l)$ is exactly the spectrum $\sigma_+(\Psi_0)$. Hence, we
  also have $|A_l|=n$ since
  $\sigma_+(\Psi_0)\cap\sigma_+(\Psi_l)=\emptyset$ by
  \eqref{eq:distinct-again}. Without loss of generality we may assume
  that $n\geq 2$. (Otherwise, the assertion of the theorem is
  obvious.)  Applying \eqref{eq:spec-equal} (with $i=1$ and $j=2$) we
  see that $A_1=A_2$ and thus $\sigma_+(\Psi_1)=\sigma_+(\Psi_2)$,
  which contradicts the assumptions of the theorem.

  Next, recall that by \eqref{eq:distinct-again} the sets $A_i$ are
  disjoint for all $i$. Fix $j$.  By \eqref{eq:spec-equal}, each set
  $A_i$ with $i\neq j$ is mapped one-to-one into $B_j$ and the images
  are also disjoint. Since $|B_j|<n$ and the number of the sets $A_i$
  with $i\neq j$ is $n$ we see that one of the sets $A_i$ must be
  empty.  Moreover, since this is true for every $j$, there must be at
  least two such sets. Indeed, assume that there is only one such set,
  say, $A_{0}=\emptyset$ but $|A_i|\geq 1$ when $i= 1,\ldots, n$. Then
$$
|B_0|\geq |A_1|+\ldots + |A_n|\geq n,
$$
which is impossible.

Thus we may assume without loss of generality that $A_0=\emptyset$ and
$A_1=\emptyset$. Then, by the definition of $A_0$, every point of
$\sigma_+(\Psi_0)$ is also a point of $\sigma_+(\Phi_0)$, i.e.,
$\sigma_+(\Psi_0)\subset \sigma_+(\Phi_0)$. Furthermore
$|\sigma_+(\Psi_0)|=n\geq |\sigma_+(\Phi_0)|$ since the elements of
$\sigma(\Psi_0)$ are distinct. It follows that
$\sigma_+(\Phi_0)=\sigma_+(\Psi_0)$ and, similarly,
$\sigma_+(\Phi_1)=\sigma_+(\Psi_1)$.  Furthermore, by
\eqref{eq:spec-equal}, $B_0=B_1$. Assume that this set is non-empty
and denote by $X\subset \sigma_+(\Phi_0)\cap\sigma_+(\Phi_1)$ the set
of the first components of $B_0=B_1$. Clearly, we also have
$X\neq \emptyset$.  Hence,
$\sigma_+(\Psi_0)\cap\sigma_+(\Psi_1)\supset X\neq \emptyset$ which is
impossible. Therefore, $B_0=\emptyset$ and
$\hsi\big(\Phi_0\oplus\Psi_0^{-1}\big)=\emptyset$.

To summarize, we have shown that the decorated spectra
\eqref{eq:spec-equal} are empty. Since all fixed points $y_i$ of
$R_\varphi$ are elliptic and the eigenvalues are distinct; the fixed
points $x_i$ of $\varphi$ are also elliptic. Therefore,
$\mult_{-1}(\Phi_i)=0$ and, by \eqref{eq:mult-index},
  $$
   \loop\big(\Phi_i\oplus\Psi_i^{-1}\big)=s(1)=0,
  $$
  where $s(1)=0$ due to the assumption that $\varphi$ is balanced. In
  a similar vein,
  $$
   s(k)=\loop\big(\Phi_i^k\oplus\Psi_i^{-k}\big)=ks(1)=0.
  $$
  Hence, $\varphi^k$ is balanced since $R_\varphi^k$ is balanced.
  This concludes the proof of the theorem.
\end{proof}

\subsection{Ghost pseudo-rotations of $\CP^2$}
\label{sec:Ghosts}
Recall from Definition \ref{def:ghost} that a strongly non-degenerate
pseudo-rotation $\varphi$ of $\CP^2$ is said to be a ghost
pseudo-rotation if its fixed points have Floquet multipliers, i.e.,
the first Krein type eigenvalues, of the form $(\lambda,\lambda)$,
$(\lambda,\eta)$ and $(\eta,\eta)$, where $\lambda$ and $\eta$ are
unit complex numbers and $\lambda=\bar{\eta}^2$. Also recall that the
Floquet multipliers do not fully determine the fixed-point data and
even the indices of the fixed points.

In this section, without fully exploring all possibilities, we will
show that for a suitable choice of fixed-point data a ghost
pseudo-rotation can satisfy all the symplectic topological conditions
from Section \ref{sec:background}. To be more specific, we prove that
for all iterates $\varphi^k$ the indices of the fixed points are
distinct and all integers of the same parity as $n$ occur as indices,
and that the marked mean index map $\cCSI$, assigning the mean index
of the point with Conley--Zehnder index $2l-n$ to $l\in \Z$, is
monotone. As a consequence, the existence of ghost pseudo-rotations
cannot be ruled out by a combinatorial argument as in the proof of
Theorem \ref{thm:balanced-4D}. We will also see that these ghost
pseudo-rotations are not balanced.

The fixed-point data also include the actions of the fixed points. For
pseudo-rotations, by Theorem \ref{thm:act-index}, the actions are
equal to the mean indices up to the factor $\pi/2(n+1)$ and a
simultaneous shift. Thus, by forcing this condition to hold, we can
have the entire fixed-point data satisfy all requirements from Section
\ref{sec:background}. Of course, this does not prove that ghost
pseudo-rotations really exist, but only that their existence does not
contradict the conditions spelled out there.

Throughout this section it is convenient to use the notation from the
proof of Theorem \ref{thm:balanced-4D}, which is slightly different
than what is used in Definition \ref{def:ghost}. Namely, we denote by
$\lambda$ and $\eta$ the eigenvalues of $D\varphi$ at the fixed points
lying in $S^1_+$ rather than the Floquet multipliers which we denote
by $\lambda_1$ and $\eta_1$. As in Remark \ref{rmk:ss} we will assume
throughout this section that all maps are semi-simple. In the proof of
Theorem \ref{thm:balanced-4D}, ghost pseudo-rotations arise in two
cases.

The first case is where one of the Floquet multipliers, say
$\lambda_1=\lambda$, is in the upper half circle $S^1_+$ and the
other, $\eta_1=\bar\eta$, is in the lower half circle, and either
$\lambda_1=\bar{\eta}_1^2$ or $\eta_1=\bar{\lambda}_1^2$. In both
situations the resonance condition from Definition \ref{def:ghost} is
satisfied with the Floquet multipliers swapped in the latter
case. Here, for the sake of brevity, we will focus on the former.

To summarize, the three fixed points of $\varphi$ have Floquet
multipliers $(\lambda,\lambda)$, $(\lambda,\bar{\eta})$ and
$(\bar{\eta},\bar{\eta})$, where $\lambda$ and $\eta$ are in $S^1_+$
and $\lambda=\eta^2$. Denote the fixed points, exactly in that order,
by $y_2$, $y_1$ and $y_0$, and set $\eta=e^{2\pi\sqrt{-1}\theta}$,
where $0<\theta<1/4$ is irrational. Hence,
$\lambda=e^{4\pi\sqrt{-1}\theta}$. Consider the short rotations
$\psi_0(t)=e^{-2\pi\sqrt{-1}\theta t}$, $t\in [0,1]$, and
$\psi_1(t)=\psi_0^{-2}(t)=e^{4\pi\sqrt{-1}\theta t}$. Then
\begin{equation}
\label{eq:ghost10}
D\varphi|_{y_i}=
\begin{cases}
\psi_0(1)\oplus\psi_0(1) &\textrm{ for } i=0,\\
\psi_0(1)\oplus\psi_1(1) &\textrm{ for } i=1,\\
\psi_1(1)\oplus\psi_1(1) &\textrm{ for } i=2,
\end{cases}
\end{equation}
where $D\varphi|_{y_i}$ on the left are viewed as time-one maps. For a
given capping of $y_i$, these conditions determine the linearized flow
only up to a loop. To specify the fixed-point data, let us assume that
for some cappings $\by_i$, we have
\begin{equation}
\label{eq:ghost11}
D\varphi|_{\by_i}=
\begin{cases}
\psi_0\oplus\psi_0 &\textrm{ for } i=0,\\
\psi_0\oplus\psi_1 &\textrm{ for } i=1,\\
\psi_1\oplus\psi_1 &\textrm{ for } i=2,
\end{cases}
\end{equation}
as elements of the universal covering of the linear symplectic
group. This assumption determines the fixed-point data (except for the
actions at $\by_i$). In our previous notation, $\by_i=\bx_i$. We will
show that this fixed-point data set does not contradict the conditions
from Section \ref{sec:background}.

The second situation where ghost pseudo-rotations arise in the proof
is when both of the Floquet multipliers are either in the upper half
circle $S^1_+$ or in the lower half circle. We will focus on the
former case. Thus $\lambda_1=\lambda$ and $\eta_1=\eta$, and
$\lambda=\bar{\eta}^2$ where $\lambda$ is the closest of the two
points to $1\in S^1$. The three fixed points $y_2$, $y_1$ and $y_0$ of
$\varphi$ have Floquet multipliers $(\lambda,\lambda)$,
$(\lambda,\eta)$ and $\eta,\eta)$. As above set
$\eta=e^{2\pi\sqrt{-1}\theta}$, where now $1/4<\theta<1/2$ and
$\theta$ is again irrational. Then
$\lambda=e^{4\pi\sqrt{-1}(1-\theta)}$. Consider the short
counter-clockwise rotations $\psi_0(t)=e^{2\pi\sqrt{-1}\theta t}$ and
$\psi_1(t)=e^{4\pi\sqrt{-1}(1-\theta) t}$, $t\in [0,1]$.

We still have \eqref{eq:ghost10}. However, \eqref{eq:ghost11} cannot
hold for any cappings of $y_i$, for the maps on the right-hand side
have now index 2 while their indices must be distinct modulo
6. Denoting by $\xi_m\in \tSp(4)$ the loop with Maslov index $2m$, we
modify \eqref{eq:ghost11} by assuming that for some cappings $\by_i$,
\begin{equation}
\label{eq:ghost21}
D\varphi|_{\by_i}=
\begin{cases}
(\psi_0\oplus\psi_0)\cdot\xi_{-2} &\textrm{ for } i=0,\\
\psi_0\oplus\psi_1 &\textrm{ for } i=1,\\
(\psi_1\oplus\psi_1)\cdot\xi_{-1} &\textrm{ for } i=2,
\end{cases}
\end{equation}
where the dot stands for the product in $\tSp(4)$. Now
$\mu(\by_0)=-2$, $\mu(\by_1)=2$ and $\mu(\by_2)=0$, i.e., in our
previous notation $\by_0=\bx_0$, $\by_1=\bx_2$ and $\by_2=\bx_1$. This
specifies a set of the fixed-point data which also turns out to be
compatible with the conditions from Section \ref{sec:background}.

\begin{Proposition}
  \label{prop:ghost1}
  With the set of the fixed-point data \eqref{eq:ghost11} and
  \eqref{eq:ghost21} for $\varphi$, the indices of the capped
  $k$-periodic orbits are in one-to-one correspondence with $2\Z$ and
  the orbits with higher index have larger action for all iterates
  $\varphi^k$, i.e., the marked mean index map $\cCSI$ is defined and
  monotone. The maps $\varphi$, if exist, are not balanced: the sum of
  the mean indices is $6\theta$ for \eqref{eq:ghost11} and $-6\theta$
  for \eqref{eq:ghost21}.
\end{Proposition}

This shows, in particular, that as stated above the fixed-point data
\eqref{eq:ghost11} and \eqref{eq:ghost21} do not contradict the
conditions from Section \ref{sec:background}. (We note that not every
modification in \eqref{eq:ghost21} produces such fixed-point data.)

\begin{proof} The argument is essentially a rather long calculation
  and here we only briefly outline the key steps starting with
  \eqref{eq:ghost11}. By the definition and additivity of the
  Conley--Zehnder index (see, e.g., \cite{SZ}), we have
\begin{eqnarray*}
  \mu\big(\by^k_0\big) &=& -2-4\lfloor k\theta\rfloor,\\
  \mu\big(\by^k_1\big) &=&
  2\lfloor 2k\theta\rfloor-2\lfloor k\theta\rfloor,\\
  \mu\big(\by^k_2\big) &=& 2+ 4\lfloor 2k\theta\rfloor
\end{eqnarray*}
for \eqref{eq:ghost11}, and
\begin{eqnarray*}
\hmu\big(\by^k_0\big) &=& -4k\theta,\\
\hmu\big(\by^k_1\big) &=& 2k\theta,\\
\hmu\big(\by^k_2\big) &=& 8k\theta.
\end{eqnarray*}

Likewise, for \eqref{eq:ghost21} the iterated indices are
\begin{eqnarray*}
\mu\big(\by^k_0\big) &=& 4\lfloor k\theta\rfloor+2-4k,\\
  \mu\big(\by^k_1\big) &=&
  2\lfloor k\theta\rfloor+2\lfloor 2k(1-\theta)\rfloor,\\
\mu\big(\by^k_2\big) &=& 4\lfloor 2k(1-\theta)\rfloor+2-2k,
\end{eqnarray*}
and 
\begin{eqnarray*}
\hmu\big(\by^k_0\big) &=& 4k(\theta-2),\\
\hmu\big(\by^k_1\big) &=& 2k(2-\theta),\\
\hmu\big(\by^k_2\big) &=& 4k(1-2\theta).
\end{eqnarray*}

Directly examining these expressions, it is not hard to see that in
both cases for every $k$ all three indices $\mu\big(\by^k_i\big)$ are
distinct modulo 6 and that $\mu\big(\by^k_i\big)>\mu\big(\by^k_j\big)$
if and only if $\hmu\big(\by^k_i\big)>\hmu\big(\by^k_j\big)$.
Recapping an orbit $\by_i^k$ results in adding an integer $6l$ to both
the index and the mean index. Again, examining the above formulas, one
can show that capped orbits in $\bPP_k(\varphi)$ with higher index
have larger mean index, i.e., the marked mean index map $\cCSI$ is
monotone. (These calculations are straightforward but tedious and we
omit them.)  Setting $k=1$ in the expressions for the mean indices, we
conclude that the sum of the mean indices of the orbits $\by_i$ is
$6\theta$ for \eqref{eq:ghost11} and $-6\theta$ for
\eqref{eq:ghost21}, i.e., the maps $\varphi$, if exist, are not
balanced.
\end{proof}

\begin{Remark}
  \label{rmk:SQ} There is a slight chance that the existence of ghost
  pseudo-rotations can be ruled out by using the equivariant
  pair-of-pants product $\wp$ introduced in \cite{Se:Eq} and the
  quantum Steenrod square $\QS$ from \cite{Wi1,Wi2}. Namely one can
  try to compare $\QS$ for $\CP^2$ with the behavior of the fixed
  points under iterations, via $\wp$. However, there is no
  contradiction in the first approximation and this approach is
  unlikely to work because it is unclear how to explicitly determine
  higher order terms (in the generator $\hh$ of
  $\H^1(\RP^{\infty};\Z_2)$) in the equivariant PSS-map without very
  specific information about $\varphi$; cf.\ \cite[Example
  3.12]{CiGG:IMRN}. (In contrast, calculating the non-equivariant
  PSS-map is straightforward for pseudo-rotations of $\CP^n$ because a
  cohomology class is completely determined by its degree.)
\end{Remark}

\section{Proof of the index divisibility theorem}
\label{sec:index-div-pf}

Our goal in this section is to prove the index divisibility theorem,
i.e., Theorem \ref{thm:index-div} (or, Theorem
\ref{thm:index-div-intro} from the introduction). Throughout the proof
we keep the notation and conventions from Section \ref{sec:index}.
Furthermore, set $\mu_k:=\mu\big(\Phi^k\big)$ and
$\mu'_k=\mu_{k+1}-\mu_k$. These sequences are additive with respect to
the direct sum. When it is essential to emphasize the role of the map
$\Phi$, we will write $\mu_k(\Phi)$ and $\mu'_k(\Phi)$.  For every
$\lambda\in\sigma_+(\Phi)$, we define the \emph{logarithmic
  eigenvalue} $\theta\in (0,\,1)$ by
$$
\lambda=\exp(\pi\sqrt{-1}\theta).
$$
Since $\Phi$ is strongly non-degenerate, $\theta\not\in\Q$.

Essentially by the definition of the Conley--Zehnder index (see, e.g.,
\cite{SZ}),
\begin{equation}
\label{eq:difference}
\mu'_k=\loop(\Phi)+\mult_{-1}(\Phi)+2\sum_{\lambda\in\sigma_+(\Phi)}
a_\lambda(k)\sgn_\lambda(\Phi),
\end{equation}
where 
\begin{equation}
\label{eq:a(k)}
a_\lambda(k)=
\begin{cases}
  0 &\textrm{ when } \lfloor
  (k+1)\theta/2\rfloor=\left\lfloor k\theta/2\right\rfloor,\\
  1 &\textrm{ when } \left\lfloor
    (k+1)\theta/2\right\rfloor=\left\lfloor k\theta/2\right\rfloor+1.
\end{cases}
\end{equation}
In the latter case we say that the eigenvalue $\lambda$ \emph{jumps}
at $k$.

The implication (b)$\Rightarrow$(a) from Theorem \ref{thm:index-div}
is an immediate consequence of \eqref{eq:difference}. The rest of this
section comprises the proof of the main assertion of the theorem, the
implication (a)$\Rightarrow$(b): the fact that (i) and (ii) hold
whenever $2l\mid \mu'_k$ for all $k\in\N$.

First, note that by \eqref{eq:difference} and \eqref{eq:a(k)} we have
$$
\mu'_1(\Phi) = \loop(\Phi)+\mult_{-1}(\Phi)
$$
which proves (i) and, as a consequence, the assertion in the case
where the map $\Phi(1)$ has no elliptic part, i.e., $\Phi_e=0$ in the
decomposition $\Phi=\phi\circ(\Phi_h\oplus \Phi_{-h}\oplus \Phi_e)$
from Section \ref{sec:index}. This also shows that the assertion holds
for $\Phi$ if and only if it holds for $\Phi_e$.  Thus in what follows
we can assume that $\Phi$ has no hyperbolic and loop parts, i.e.,
$\Phi=\Phi_e$, and, in particular,
\[
  \loop(\Phi)+\mult_{-1}(\Phi)=0.
\]
Then the iterations $\Phi^k$ also have no hyperbolic part, although
the loop part of the iterated map may be (and usually is) non-trivial.

The idea of the proof is that in this situation there exists an
eigenvalue $\lambda$ and an iteration $k$ such that only $\lambda$
jumps at $k$, and hence only $\lambda$ contributes to the change of
the index between $k$ and $k+1$. Then, $\mu'_k = \sgn_\lambda(\Phi)$
by \eqref{eq:difference}. (In fact, this is true for an infinite
sequence of iterations.) Finding such an eigenvalue is the central
part of the proof, and we believe that the statement is not true for a
particular fixed eigenvalue. The difficulty in the proof of this
assertion, which is in essence a statement about the geometry of
$\T^r$, lies in the fact that the eigenvalues might satisfy resonance
relations and in our setting they necessarily do.

The map $\Phi =\Phi_e$ decomposes into the direct sum of maps
$\Phi(\lambda)$ with eigenvalues $\lambda$ and $\bar{\lambda}$, where
$\lambda\in\sigma_+(\Phi)$:
\[
    \Phi=\bigoplus_{\lambda\in\sigma_+(\Phi)}\Phi(\lambda).
\]
Clearly, $\sgn_\lambda(\Phi(\lambda))=\sgn_\lambda(\Phi)$ and
$\sgn_{\lambda'}(\Phi(\lambda))=0$ when $\lambda'\neq \lambda$.

The implication (a)$\Rightarrow$(b) holds for $\Phi$ and one of the
maps $\Phi(\lambda)$ if and only if it holds for $\Phi(\lambda)$ and
the sum
$$ \Psi:=\bigoplus_{\lambda'\neq
  \lambda}\Phi(\lambda')
$$
of the remaining maps. The idea of the proof is to show that
\begin{equation}
\label{eq:div0}
l\mid\sgn_\lambda(\Phi)
\end{equation}
for some eigenvalue $\lambda$.  Once \eqref{eq:div0} is established,
$l\mid \mu'_k\big(\Phi(\lambda)\big)$ because (b) implies (a). Since
$$
\mu'_k(\Phi) =\mu'_k(\Psi)+\mu'_k\big(\Phi(\lambda)\big),
$$
where as above $\Phi=\Psi\oplus\Phi(\lambda)$, to prove the assertion
for $\Phi$ it is enough to establish it for $\Psi$. Now the result
follows by induction on dimension.

Turning to the actual proof, we first need to introduce some
terminology and notation.  Let us order the positive unit spectrum
$\sigma_+(\Phi)$ in an arbitrary way:
$$
\sigma_+(\Phi)=\{\lambda_1,\dotsc,\lambda_r\},
$$
and consider the eigenvalue ``vector''
\begin{equation}
\label{eq:lambda-vector}
\vec{\lambda}=\vec{\lambda}(\Phi)=\big(\lambda_1,\dotsc,
\lambda_r\big)\in\T^{r}:=(S^1)^{r}.
\end{equation}
Denote by $\Gamma(\Phi)$ or just $\Gamma$ the closure of the positive
orbit
$$
\Lambda=\{\vec{\lambda}^k\mid k\in\N\}\subset \T^r.
$$
This is a subgroup of $\T^{r}$. Let $\Gamma_0$ be the connected
component of the identity of $\Gamma$.

Denoting the coordinates on $\T^r$ by $(z_1,\dotsc,z_r)$, consider the
codimension one sub-tori $\T^{r-1}_i$ given by the condition
$z_i=1$. It is easy to see from the requirement that $\Phi$ is
strongly non-degenerate (i.e., all $\theta_i\not\in\Q$) that the group
$\Gamma$ intersects the submanifolds $\T^{r-1}_i$ transversely. We
co-orient $\T^{r-1}_i$ via the positive (clockwise) orientation of
$z_i$. Furthermore, let us call a formal linear combination of closed
co-oriented submanifolds of $\T^r$ with integer coefficients a
\emph{cycle}. For a codimension-one cycle $T$ and an oriented path
$\eta$ in $\T^r$ with end-points outside $T$ the intersection index
$\left<\eta,\,T\right>\in\Z$ is defined in an obvious way.

To $\Phi$, we associate the cycle
$$
T=\sum T_i,\textrm{ where } T_i=\sgn_{\lambda_i}(\Phi)\T^{r-1}_i,
$$
which we call the \emph{index cycle} of $\Phi$.

Finally, consider the oriented path
$$
A=\big\{\big(e^{\pi\sqrt{-1}\theta_1 t},\ldots,
e^{\pi\sqrt{-1}\theta_r t}\big)\mid t\in [0,\,1]\big\},
$$
which we will refer to as the \emph{generating arc}.  Then the arc
$\vec{\lambda}^k A$, where we use multiplicative notation for the
group operation in $\T^r$, connects $\vec{\lambda}^k$ to
$\vec{\lambda}^{k+1}$. The points $\vec{\lambda}^k$, $k\in\N$, are
never in $T$ due to the strong non-degeneracy of $\Phi$ and thus the
intersection index $\big<\vec{\lambda}^k A,\,T\big>\in\Z$ is
defined. Then \eqref{eq:difference} is simply the fact that
\begin{equation}
  \label{eq:mu-intersect}
2\big<\vec{\lambda}^k A,\,T\big>=\mu'_k,
\end{equation}
where we have assumed that $\Phi$ does not have the loop and
hyperbolic parts. The individual terms in \eqref{eq:difference} can be
interpreted in a similar vein. Namely, with $a_{\lambda_i}(k)$ defined
by \eqref{eq:a(k)}, we have
$$
\big<\vec{\lambda}^k
A,\,T_i\big>=a_{\lambda_i}(k)\sgn_{\lambda_i}(\Phi).
$$

The proof of the theorem hinges on two lemmas.

\begin{Lemma}
  \label{lemma:intersect}
  Assume that $2l\mid \mu'_k$ for all $k\in\N$. Then for every
  oriented path $\eta$ in $\Gamma$ with end points outside $T$, the
  intersection index $\left<\eta,\,T\right>$ is divisible by $l$:
\begin{equation}
  \label{eq:div-by-l}
l\mid \left<\eta,\,T\right>.
\end{equation}
\end{Lemma}

\begin{Remark} Observe that this lemma must hold if the theorem is
  true: the cycles $\T_i^{r-1}$ enter $T$ with coefficients divisible
  by $l$. Moreover, \eqref{eq:div-by-l} must be satisfied for all
  paths $\eta$ in $\T^r$, but not just in $\Gamma$, with end points
  outside $T$.
\end{Remark}

Postponing the proof of the lemma, set
$$
C_i=\Gamma\cap \T^{r-1}_i.
$$
As has been pointed out above, this intersection is transverse. (Note
that the subgroup $C_i$ can have several connected components even
when $\Gamma$ is connected.)

\begin{Lemma}
  \label{lemma:components}
  There exists $i$ such that $C_i$ is not entirely contained in the
  union of the subgroups $C_j$ with $j\neq i$:
  \begin{equation}
    \label{eq:distinct}
  C_i\not\subset\bigcup_{j\neq i}C_j.
\end{equation}
\end{Lemma}

\begin{Remark}
  This lemma depends only on the assumptions that the components
  $\lambda_i$ of $\vec{\lambda}$ are distinct and that
  $\theta_i\in (0,\,1)$ for all $i$. (However, the eigenvalues of
  $\Phi$ need not be distinct.) Note also that, as will be clear from
  the proof, while those $i$ for which \eqref{eq:distinct} holds can
  be explicitly described, \eqref{eq:distinct} does not need to be
  satisfied for all $i$. In fact, we can have $C_i\subset C_j$ for
  some $i\neq j$, and then of course \eqref{eq:distinct} does not hold
  for this $i$.
\end{Remark}

The theorem readily follows from these two lemmas. Namely, let $i$ be
as in Lemma \ref{lemma:components}. Pick a short path $\eta$ in
$\Gamma$ transverse to $C_i$, intersecting $C_i$ at one point and not
intersecting any $C_j$ with $j\neq i$. Such a path exists since the
complement to the union of $C_j$, $j\neq i$, in $C_i$ is non-empty.

Then, by Lemma
\ref{lemma:intersect},
$$
l\mid \left<\eta,\,T\right>=\left<\eta,\,T_i\right>=\pm
\sgn_{\lambda_i}(\Phi).
$$
Splitting off $\Phi(\lambda_i)$ as described above, we reduce the
dimension and the theorem follows by induction.  To complete the
proof, it remains to prove the lemmas.

\begin{proof}[Proof of Lemma \ref{lemma:intersect}] Let us equip
  $\Gamma$ with the metric induced by the standard flat metric on
  $\T^r$.  It suffices to prove the lemma for short geodesics $\xi$
  (of length less than some $\eps>0$ to be specified later) connecting
  points of the orbit $\Lambda$. Indeed, any path $\eta$ can be
  arbitrarily well approximated by broken geodesics with segments
  $\xi_q$ of this type. Thus
  $\left<\eta,\,T\right>=\sum \left<\xi_q,\,T\right>$, where every
  term on the right is divisible by $l$. (In fact, proving the lemma
  for such short geodesics $\xi$ would be sufficient for our
  purposes.)

  We show that $l\mid \left<\xi,\,T\right>$ in several steps. First
  consider a path $\alpha$ in $\T^r$ -- an iterated arc -- obtained by
  concatenating several adjacent copies of the generating arc $A$:
  \begin{equation}
    \label{eq:iterated_arc}
    \alpha=\vec{\lambda}^k A\cup \vec{\lambda}^{k+1}A\cup\ldots\cup
    \vec{\lambda}^{m} A.
\end{equation}
Then, by \eqref{eq:mu-intersect},
  $$
  2\left<\alpha,\,T\right>=\mu'_k+\mu'_{k+1}+\ldots+\mu'_{m},
  $$
  and hence
  \begin{equation}
    \label{eq:div-alpha}
    l\mid \left<\alpha,\,T\right>.
\end{equation}

Next, fix a small neighborhood $V$ of $\vec{\lambda}$ disjoint from
$T$ and set $\eps>0$ to be the radius of $V$, i.e., the supremum of
the length of a geodesic in $V$ starting at $\vec{\lambda}$.  Pick a
point $\vec{\lambda}^k\in V\cap\Lambda$ and let $\alpha$ be a path as
above connecting $\vec{\lambda}$ to $\vec{\lambda}^k$. We complete
$\alpha$ to a loop $\gamma$ by concatenating it with a geodesic
$\zeta$ in $V$ connecting its end points. Then
  $$
  \left<\gamma,\,T\right>=\left<\alpha,\,T\right>,
  $$
  since $V$ is disjoint from $T$ and thus
  $\left<\zeta,\,T\right>=0$. Therefore, by \eqref{eq:div-alpha},
  \begin{equation}
    \label{eq:div-gamma}
  l\mid \left<\gamma,\,T\right>.
\end{equation}
Moreover, since the intersection index $\left<\gamma,\,T\right>$
depends only on the homology class of $\gamma$, the same is true for
any loop $\gamma'$ obtained from $\gamma$ by parallel transport in
$\T^{r}$.

Let $\xi$ be a geodesic in $\Gamma$ of length less than $\eps$
connecting $\vec{\lambda}^k$ to $\vec{\lambda}^{k'}$ for some $k'$,
where without loss of generality we can assume that $k<k'$. Let us
connect $\vec{\lambda}^k$ to $\vec{\lambda}^{k'}$ by the iterated arc
$\alpha'$ defined by \eqref{eq:iterated_arc} with
$m=k'$. Concatenating $\alpha'$ with the geodesic $-\xi$ (the reversed
orientation) we obtain a loop $\gamma'$ in $\T^r$.

Consider the loop $\gamma=\vec{\lambda}^{-k+1}\gamma'$. Then
$\left<\gamma',\,T\right>=\left<\gamma,\,T\right>$ and thus
  $$
  \left<\xi,\,T\right>=
  \left<\gamma',\,T\right>-\left<\alpha',\,T\right>=
  \left<\gamma,\,T\right>-\left<\alpha',\,T\right>.
  $$
  By \eqref{eq:div-alpha} and \eqref{eq:div-gamma}, both terms on the
  right-hand side are divisible by $l$. Therefore,
  $l\mid \left<\xi,\,T\right>$, which concludes the proof of the
  lemma.
  \end{proof}

  \begin{proof}[Proof of Lemma \ref{lemma:components}] The argument is
    carried out in three steps.

    \emph{Step I.} To set the stage for dealing with more complicated
    situations, let us first consider the case where $\Gamma$ is
    connected. Then $\Gamma$ contains a dense one-parameter subgroup
    \begin{equation}
      \label{eq:1-param}
      \vec{\alpha}^t:=\big\{\big(e^{\pi\alpha_1\sqrt{-1}t},\dotsc, 
      e^{\pi\alpha_r\sqrt{-1}t}\big)\mid t\in\R\big\}.
  \end{equation}
  Observe that all coefficients $\alpha_i$ are necessarily
  distinct. Indeed, assume otherwise: e.g., $\alpha_1=\alpha_2$. Then,
  since the one-parameter subgroup is dense in $\Gamma$, the
  coordinates $z_1$ and $z_2$ agree for all points in $\Gamma$, which
  is impossible because $\lambda_1\neq \lambda_2$.

  Moreover, we claim that the absolute values $|\alpha_i|$ are also
  distinct. Again, arguing by contradiction, assume that, e.g.,
  $\alpha_1=-\alpha_2$. Then $z_1=\bar{z}_2$ at every point of
  $\Gamma$. However, we have
  $\lambda_i=\exp\big(\pi\sqrt{-1}\theta_i\big)$, where $0<\theta_i<1$
  for all $i$, and hence $\lambda_i\neq \bar{\lambda}_j$ for any $i$
  and $j$.

  Pick $\alpha_i$ with the largest absolute value and set
  $t=2/\alpha_i$. Then, as is easy to see,
    \[
      \vec{\alpha}^t\in C_i,\textrm{ but }\vec{\alpha}^t\not\in
      C_j\textrm{ when } j\neq i.
    \]

    \emph{Step II.} Here we focus on the case where $\Gamma$ is
    one-dimensional. This is the crucial step of the proof. When
    $\Gamma$ is connected, Step 1 applies. Thus we can assume that
    $\Gamma/\Gamma_0$ is non-trivial. (This group is automatically
    cyclic -- it is generated by the image of the topological
    generator $\vec{\lambda}$ of $\Gamma$ -- but we do not need this
    fact.) We start with several observations.

    First, note that it suffices to only consider the case where
    $\Gamma/\Gamma_0$ is a cyclic subgroup $\Z_p$ of prime order $p$.
    Indeed, pick a subgroup isomorphic to $\Z_p$ in $\Gamma/\Gamma_0$
    for some prime $p$ and let $\Gamma'$ be its inverse image in
    $\Gamma$. Then, since every connected component of $\Gamma'$ is
    also a connected component of $\Gamma$, the lemma holds for
    $\Gamma$ whenever it holds for $\Gamma'$. From now on we will
    assume that $\Gamma/\Gamma_0\cong \Z_p$ where $p$ is
    prime. Furthermore, without loss of generality we can also assume
    that $\vec{\lambda}$ projects to $1\in\Z_p$. We denote its
    connected component by $\Gamma_1$.

    Next, we need to find a convenient parametrization of $\Gamma$.
    The connected component of the identity $\Gamma_0$ of $\Gamma$ is
    a one-parameter subgroup of the form \eqref{eq:1-param}, where now
    $\alpha_i\in\Z$ for all $i$. However, in contrast with Step I, the
    map $t\mapsto \vec{\alpha}^t$ from $\R$ to $\Gamma_0$ is not
    one-to-one. Without loss of generality, we may assume that
    $\gcd(\alpha_1,\ldots,\alpha_r) =1$.  Then this parametrization of
    $\Gamma_0$ is a group homomorphism with kernel $2\Z$. Since
    $\Gamma$ and hence $\Gamma_0$ are transverse to $\T^{r-1}_i$, we
    have $\alpha_i\neq 0$ for all $i$.

    Furthermore, we claim that there exists $\vec{\beta}\in \Gamma_1$
    such that $\vec{\beta}^p=1$ and the map
    \begin{equation}
      \label{eq:param}
      P\colon\R/2\Z\times \Z_p\to \Gamma, \quad (t,k)
      \mapsto \vec{\alpha}^t\vec{\beta}^k
  \end{equation}
  is an isomorphism. Indeed, by construction,
  $\vec{\lambda}^p\in\Gamma_0$. Since $\Gamma_0\cong S^1$, there
  exists, $\vec{\zeta}\in\Gamma_0$ such that
  $\vec{\zeta}^p=\vec{\lambda}^p$, and it suffices to set
  $\vec{\beta}=\vec{\zeta}^{-1}\vec{\lambda}\in \Gamma_1$. The map
  \eqref{eq:param} is the sought parametrization of $\Gamma$ and it is
  routine to check that $P$ is indeed a group isomorphism and
  $\Gamma_0$ is the image of $\R/2\Z\times \{0\}$. Since
  $\vec{\beta}^p=1$, we have
  \[
    \vec{\beta}=\big(e^{2\pi\alpha_1\sqrt{-1}q_1/p}, \dotsc,
    e^{2\pi\alpha_r\sqrt{-1}q_r/p}\big), \textrm{ where } 0\leq q_i< p
    \textrm { for all } i=1,\ldots, r.
  \]
  Note also that $\vec{\beta}$ satisfying the above conditions is not
  unique. The ambiguity is up to a $p$-th root of unity in $\Gamma$,
  i.e., up to multiplication by $\vec{\alpha}^{2m/p}$, $m\in\Z$. (We
  do not need this fact.)

  Observe for a future use that we can take now any element
  $\vec{\lambda}=\vec{\alpha}^t\vec{\beta}$, with $t\not\in\Q$, as a
  topological generator of $\Gamma$. Then
  \begin{equation}
    \label{eq:lambda_j}
  \lambda_j=e^{\pi\sqrt{-1}(\alpha_jt+2q_j/p)}.
\end{equation}

Consider the intersections
$$
C_i^0=\Gamma_0\cap \T_i^{r-1}\subset C_i.
$$
This is the group of the roots of unity in $\Gamma_0\cong S^1$ of
degree $|\alpha_i|$, i.e., the unique cyclic subgroup of order
$|C_i^0|=|\alpha_i|$. We partially order the set of these subgroups by
inclusion, which is equivalent to partially ordering the collection
$\{\alpha_i\}$ by divisibility: $\alpha_i$ is considered to be greater
than or equal to $\alpha_j$ if $\alpha_j$ divides $\alpha_i$ or
equivalently $C^0_j\subset C^0_i$. (When $\Gamma$ is connected,
$\vec{\lambda}=\vec{\alpha}^t$ for some $t\not\in\Q$, and it follows
that all $|\alpha_i|$ are distinct. Hence, all groups $C_j=C_j^0$ are
distinct and cyclic. As a consequence, a maximal subgroup $C_i$ is not
contained in the union of the other subgroups $C_j$. This is a
slightly simplified version of the argument in Step I proving Lemma
\ref{lemma:components} in the case where $\Gamma$ is connected and
$\dim\Gamma=1$.)

Next observe that the natural map
$$
C_i/C_i^0\to \Gamma/\Gamma_0
$$
is necessarily onto, and hence the sequence
\begin{equation}
  \label{eq:C-seq}
  0\to C_i^0\to C_i\to \Gamma/\Gamma_0\to 0
\end{equation}
is exact. The reason is that the intersections of the connected
components of $\Gamma$ with $\T_i^{r-1}$ are disjoint, since the
components are disjoint, and have the same cardinality since the
components represent the same homology class. (Alternatively, these
intersections are all obtained by parallel transport from each other.)
Then the intersection of every component with $\T_i^{r-1}$ is
non-empty because $\Gamma_0\cap \T_i^{r-1}$ contains the unit
in~$\T^r$.

Without the requirement that $\Gamma$ is connected, the proof hinges
on two claims. The first one, independent of the background assumption
that $\Gamma/\Gamma_0$ has prime order, asserts that while in general
the groups $C_i^0$ may coincide, the full groups $C_i$ are necessarily
distinct.

\begin{Claim}
  \label{claim:distinct}
  For any $s\neq j$, we have $C_s\neq C_j$.
\end{Claim}

The second claim which is the key to the proof shows that $C_i$ is
almost always cyclic. This claim fully relies on the assumption that
$\Gamma/\Gamma_0\cong\Z_p$.

    \begin{Claim}
      \label{claim:cyclic}
      The group $C_i$ is \emph{not} cyclic if and only if
      $p \mid \alpha_i$ and $q_i=0$.
    \end{Claim}

    The key point here is that $C_i$ is cyclic if $q_i\neq 0$. The
    rest of the statement, not used in our argument, is easy to prove
    and included only for the sake of completeness.
    
    Assuming both claims let us prove Lemma \ref{lemma:components}
    when $\dim \Gamma=1$.  Let $C^0_i$ be a maximal
    subgroup with respect to the inclusion. Without loss of
    generality, we may require that $i=r$. We will show that $C_r$ is
    not covered by other groups $C_j$ or there exists another group
    $C_s$ such that $C^0_s=C^0_r$ and $C_s$ is not covered.

    By maximality, $C^0_r$ is not contained in any other subgroup
    $C_j^0$ unless $C_j^0=C_r^0$ or, equivalently,
    $\alpha_r\nmid\alpha_j$ unless $|\alpha_r|=|\alpha_j|$.  Since
    $C^0_r$ is cyclic, we can have
    \begin{equation}
      \label{eq:inclusion}
    C_r\subset \bigcup_{j\neq r} C_j
  \end{equation}
  and, as a consequence (by intersecting with $\Gamma_0$),
  $$
  C_r^0\subset \bigcup_{j\neq r} C_j^0
  $$
  only when $C^0_r\subset C_s^0$ for some $s<r$. Because $C^0_r$ is
  maximal, this is equivalent to $C^0_r= C_s^0$, i.e.,
  $\alpha_s=\pm\alpha_r$. In other words, $C_r^0$ (and hence $C_r$) is
  not contained in the union of other subgroups if it is a unique
  maximal subgroup of order $|\alpha_r|$, and then the proof is
  finished.

  Hence, to continue, we can assume that there exist at least two
  maximal subgroups $C^0_s=C^0_r$ of the same order
  $|\alpha_r|=|\alpha_s|$. In this case we may not have both $q_r$ and
  $q_s$ equal to zero. Indeed, if $\alpha_r=\alpha_s$ we may not have
  $q_r=q_s$, for then we would have $\lambda_r=\lambda_s$ by
  \eqref{eq:lambda_j}. If $\alpha_r=-\alpha_s$, we may not have
  $q_r=0=q_s$, for then we would have
  $\lambda_r=\bar{\lambda}_s$. Without loss of generality, we can
  assume that $q_s\neq 0$. By Claim \ref{claim:cyclic}, $C_s$ is
  cyclic.

  Finally, let us show that $C_s$ is not contained in the union of
  $C_j$, $j\neq s$. To see this, assume the contrary:
  \eqref{eq:inclusion} is satisfied with $r$ replaced by $s$. Then
  $C_s\subset C_j$ for some $j\neq s$ since $C_s$ is cyclic. We claim
  that in fact $C_s=C_j$. Indeed, from $C_s\subset C_j$ we infer that
  $C_s^0=C_j^0$ because $C_s^0$ is maximal. As a consequence,
  $|C_j|=|C_s|$ and hence $C_s=C_j$, which is impossible by Claim
  \ref{claim:distinct}.

  This proves Lemma \ref{lemma:components} in the setting of Step II
  modulo Claims \ref{claim:distinct} and \ref{claim:cyclic}, and it
  remains to prove the claims.

  \begin{proof}[Proof of Claim \ref{claim:distinct}]
  We will prove more: $C_s^1\cap C_j^1=\emptyset$ when $C_s^0=C_j^0$,
  where $C_i^1=C_i\cap \Gamma_1$. Thus assume that $C_s^0=C_j^0$ or
  equivalently $|\alpha_s|=|\alpha_j|$.  Set
   $$
   \pi=(\pi_{s},\pi_j)\colon \T^r\to \T^2
  $$
  to be the projection to the product of the two coordinate circles
  $S^1_{s}$ and $S^1_j$. These circles intersect only at the origin in
  $\T^2$. Thus, if $C_{s}^1\cap C_j^1$ is non-empty, we have
  $$
  \pi\big(C_{s}^1\cap C_j^1 \big)\subset \pi\big(C_{s}^1\big) \cap
  \pi\big(C_j^1\big)\subset
  \big(S^1_{s}\times\{1\}\big)\cap\big(\{1\}\times S^1_j\big)
  =\{(1,1)\}
  $$
  which is exactly the origin in $\T^2$.  As a consequence,
  $\pi(\Gamma_1)$ contains the origin $\{(1,1)\}$. On the other hand,
  $\pi(\Gamma_1)$ is a parallel translation of a subgroup and, to be
  more precise, of $\pi(\Gamma_0)$. A parallel translation of a
  subgroup is necessarily a subgroup when it contains the unit
  element.  Hence, $\pi(\Gamma_1)$ is also a subgroup and
  $$
  \pi(\Gamma_1)=\pi(\Gamma_0),
  $$
  and furthermore $\pi(\Gamma_0)=\pi(\Gamma)$ since
  $\vec{\lambda}\in\Gamma_1$ is a topological generator.

  In particular, $\pi(\vec{\lambda})\in\pi(\Gamma_0)$. As a
  consequence, for some $t\in\R$ we have
  $$
  (\lambda_{s},\lambda_j)=\pi\big(\vec{\lambda}\big)
  =\pi\big(\vec{\alpha}^t\big)= \big(e^{\pi\alpha_{s}\sqrt{-1}t},
  e^{\pi\alpha_j\sqrt{-1}t}\big),
  $$
  where $|\alpha_j|=|\alpha_s|$. This implies that
  $\lambda_{s}=\lambda_j$ if $\alpha_{s}=\alpha_j$, which is
  impossible since all components of $\vec{\lambda}$ are distinct. If
  $\alpha_{s}=-\alpha_j$, we have $\lambda_{s}=\bar{\lambda}_j$ which
  is also impossible since $\theta_{s}$ and $\theta_r$ are both in
  $(0,\,1)$. This proves Claim \ref{claim:distinct}.
\end{proof}

   \begin{proof}[Proof of Claim \ref{claim:cyclic}]
     Throughout the proof we will use additive notation and suppress
     the subscript $i$ in the notation. Thus, $C=C_i$,
     $\alpha=\alpha_i$, etc. By \eqref{eq:C-seq}, we have the short
     exact sequence
      $$
      0\to C^0\to C\to \Z_p\to 0,
      $$
      where $C^0$ is cyclic of order $|\alpha|$. It follows that $C$
      is automatically cyclic when $\gcd(p,\alpha)=1$, i.e.,
      $p\nmid\alpha$ since $p$ is prime.

      Let us identify $C$ with its inverse image in
      $\R/2\Z\times \Z_p$. Then $C$ is determined by the condition
      \begin{equation}
        \label{eq:tk}
      \frac{t}{2}\alpha +\frac{kq}{p}\in\Z.
    \end{equation}
    Here we can view $k$ as either an element of $\Z_p$ or an element
    of $\Z$, and we prefer the latter. Hence, if $q=0$ we have
    $C=C^0\times \Z_p$ since $t$ and $k$ are not related. Thus $C$ is
    not cyclic if $p \mid \alpha$.

    To finish the proof -- and this is the heart of the matter -- it
    suffices to show that $C$ is cyclic whenever $q\neq 0$. (We do not
    need to use the condition $p\nmid\alpha$ at this point.)

    Consider $g=(t,k)\in C$ such that $k \not\equiv 0 \mod p$, i.e.,
    $g\not\in C^0$. Clearly, $pg\in C^0$ (for any $g\in C$). We claim
    that $g$ is a generator of $C$ if (and only if) $pg$ is a
    generator of $C^0$. Indeed, assume that $pg$ is a generator of
    $C^0$ and denote by $G\subset C$ the subgroup generated by
    $g$. Then $C^0\subset G$. We have the exact sequence
      $$
      0\to C^0\to G\to\Z_p\to 0,
      $$
      where the map $G\to\Z_p$ is onto since $k \not\equiv 0 \mod p$
      and $p$ is prime. Thus $|G|=p|\alpha|=|C|$, and hence
      $G=C$. (The converse is not used in our argument and we omit its
      proof.)

      It remains to find $g=(t,k)$ with $k\neq 0$ such that $pg$ is a
      generator of $C^0$. By \eqref{eq:tk}, $t$ and $k$ are related by
      the condition that
      $$
      \frac{t}{2}=-\frac{kq}{p\alpha}+\frac{l}{\alpha}
      $$
      for some $l\in\Z$, and $pt$ needs to be a generator of
      $C^0\subset \R/2\Z$. Identifying $\Z_{|\alpha|}$ with
      $C^0\subset \R/2\Z$ via $s\mapsto 2s/\alpha$,
      $s\in \Z_{|\alpha|}$, we see that $pg$ is a generator of $C^0$
      if and only if $-kq+lp$ is relatively prime with $\alpha$. In
      other words, we need to find $k \not\equiv 0 \mod p$ and
      $l\in \Z$ such that $-kq+lp$ is relatively prime with
      $\alpha$. For instance, we can look for a solution $(k,l)$ of
      the equation
      $$
      -kq+lp=1.
      $$
      This equation does have a solution since $0<q<p$ and thus
      $\gcd(q,p)=1$. Moreover, $k$ is not divisible by $p$ because $1$
      is not divisible by $p$.
    \end{proof}

  This concludes Step II of the proof.
   
  \emph{Step III.} To deal with the general case, we observe that the
  lemma automatically holds for $\Gamma$ whenever it holds for a
  closed subgroup $\Gamma'$ of $\Gamma$. (We have already used this
  observation in Step II, reducing the problem to the setting where
  $\Gamma/\Gamma_0\cong\Z_p$.)

  Furthermore, it is not hard to see that one can find a point
  $\vec{\lambda}'\in\Gamma$ arbitrarily close to $\vec{\lambda}$,
  which topologically generates a one-dimensional subgroup $\Gamma'$
  of $\Gamma$.

  Indeed, the assertion is clear when $\Gamma$ is connected.  For, in
  this case, identifying $\Gamma=\Gamma_0$ with some torus
  $\T^m=\R^m/\Z^m$, it suffices to take a point $\vec{\zeta}$ with all
  irrational coordinates $(z_1,\ldots,z_m)$ such that all ratios
  $z_i/z_j$ are rational. (In other words, the union of closed
  one-parameter subgroups is dense in the torus.)  Next, let $k$ be
  the order of the cyclic group $\Gamma/\Gamma_0$, i.e., $k$ is the
  smallest positive integer such that $\vec{\lambda}^k\in
  \Gamma_0$. Arbitrarily close to $\vec{\lambda}^k$ there exists a
  point $\vec{\zeta}$ topologically generating a closed
  one-dimensional subgroup of $\Gamma_0$.  The map $z\mapsto z^k$ from
  $\Gamma_1$ to $\Gamma_0$ is onto -- it is a covering map -- and we
  can set $\vec{\lambda}'$ to be an inverse image of $\vec{\zeta}$.

  When $\vec{\lambda}'$ is sufficiently close to $\vec{\lambda}$, its
  components $\lambda'_i=\exp\big(\pi\sqrt{-1}\theta'_i\big)$ are
  distinct and all $\theta'_i$ can be taken close to $\theta_i$ and
  hence in $(0,\,1)$. Therefore, by Step II, the lemma holds for
  $\Gamma'$ and, as has been pointed out above, it then also holds for
  $\Gamma$.
  \end{proof}

\end{document}